\documentclass[12pt]{amsart}
 \usepackage[colorlinks,citecolor=blue,urlcolor=blue,linkcolor=blue]{hyperref}
\usepackage[top=1 in,bottom=1 in, left=1 in, right= 1 in, includehead,includefoot]{geometry}
\RequirePackage{amsthm,amsmath,amsfonts,amssymb}
\makeatletter
\renewcommand*\env@matrix[1][\arraystretch]{%
  \edef\arraystretch{#1}%
  \hskip -\arraycolsep
  \let\@ifnextchar\new@ifnextchar
  \array{*\c@MaxMatrixCols c}}
\makeatother

\usepackage{enumerate}

 \usepackage{framed,color,pgf,hyperref}

\newcommand{\arxivc}[1]{}
 \usepackage{tikz}
 \usetikzlibrary{patterns}
\usepackage[sort&compress]{natbib}
 \usepackage{url}
\newcommand{\comment}[1]{}
\newcommand{\commentYZ}[1]{}
\newcommand{\longcomment}[1]{}
\renewcommand{\longcomment}[1]{\ovalbox{\begin{minipage}{1.1\textwidth}\color{blue}#1\end{minipage}}}
\renewcommand{\comment}[1]{{\color{blue}\ovalbox{#1}}}
\newcommand{\hide}[1]{{\color{blue}\ovalbox{\tiny (hidden text)}}}
\def\fddto{\stackrel{\rm f.d.d.}{\Longrightarrow}}
\newcommand{\ind}{{\bf 1}}

\def\g{{\mathsf q}}
\def\p{{\mathsf p}}
\newcommand{\eqnh}{\begin{eqnarray*}}
\newcommand{\eqne}{\end{eqnarray*}}
\newcommand{\eqnhn}{\begin{eqnarray}}
\newcommand{\eqnen}{\end{eqnarray}}
\newcommand{\equh}{\begin{equation}}
\newcommand{\eque}{\end{equation}}

    \def\i{{\textnormal i}}

\def\topp#1{^{(#1)}}

\def\abs#1{\left|#1\right|}

\def\ccbb#1{\left\{#1\right\}}

\def\pp#1{\left(#1\right)}

\def\bb#1{\left[#1\right]}

\def\floor#1{\left\lfloor #1 \right\rfloor}

\def\vv#1{{\boldsymbol #1}}
\def\vvalpha{{\vv\alpha}}
\def\vvbeta{{\vv\beta}}

\def\qmwith{\quad\mbox{ with }\quad}

\def\wt#1{\widetilde{#1}}

\def\weakto{\Rightarrow}

\def\fddto{\xrightarrow{\textit{f.d.d.}}}
\renewcommand{\ind}{{\bf 1}}

 \def\re{\textnormal {Re}}
 \def\im{\textnormal {Im}}

\theoremstyle{plain}
\newtheorem{theorem}{Theorem}[section]
\newtheorem{corollary}{Corollary}[section]
\newtheorem{lemma}[theorem]{Lemma}
\newtheorem{proposition}[theorem]{Proposition}

\theoremstyle{definition}

\theoremstyle{remark}

\newtheorem{remark}{Remark}[section]

\newtheorem{assumption}{Assumption}[section]

\newcommand{\calM}{\mathcal{M}}

\def\<{\langle}
\def\>{\rangle}

\usepackage{dsfont}
\newcommand{\RR}{\mathds{R}}
\newcommand{\CC}{\mathds{C}}

\newcommand{\ZZ}{\mathds{Z}}

\newcommand{\eps}{\varepsilon}
\newcommand{\la}{\lambda}

\newcommand{\Jz}[2]{m_{#2}(#1)}
\usepackage{graphicx,fancybox,pgf,amsmath,float,amssymb}
\usepackage{fancybox,graphicx,dsfont,pgf}
\usepackage{subfigure,color,wrapfig}

\usepackage{latexsym}

\newcommand{\EE}{\mathds{E}}
\renewcommand{\Pr}{\mathds{P}}

\def\<{\langle}
\def\>{\rangle}

\def\C{ {\mathsf c}}

\newcommand{\mM}{{\vv M}}
\newcommand{\mW}{{\vv W}}
\newcommand{\mI}{{\vv I}}

\numberwithin{equation}{section}

\def\c{{\CC}}
\def\r{{\RR}}
\title{Random Motzkin paths near boundary}
\title{Limits of Random Motzkin paths  with KPZ related asymptotics}

\author{W{\l}odzimierz Bryc}
\address
{
W{\l}odzimierz Bryc\\
Department of Mathematical Sciences\\
University of Cincinnati\\
2815 Commons Way\\
Cincinnati, OH, 45221-0025, USA.
}
\email{wlodzimierz.bryc@uc.edu}

\author{Alexey Kuznetsov}
\address
{
Alexey Kuznetsov\\
Department of Mathematics and Statistics\\
York University, 4700 Keele Street
\\ Toronto, Ontario, M3J 1P3, Canada
}
\email{akuznets@yorku.ca}
\author{Jacek Weso{\l}owski}
\address{Jacek Weso{\l}owski, Faculty of Mathematics and Information Science,
Warsaw University of Technology, pl. Politechniki 1, 00-661
Warszawa, Poland}
\email{jacek.wesolowski@pw.edu.pl}


\begin{document}\sloppy
\begin{abstract}
We study Motzkin paths of length $L$ with general weights on the edges and endpoints. We  investigate the limit behavior of the initial and final segments of the random Motzkin path viewed as a pair of processes starting from each of the two endpoints as $L$ becomes large.  We  then study macroscopic limits of the resulting processes, where in two different regimes we obtain Markov processes that appeared in the description of the stationary measure for the KPZ  equation on the half-line and of conjectural stationary measure of the hypothetical KPZ fixed point on the half-line.
Our results rely on the behavior of the Al-Salam--Chihara polynomials in the neighbourhood of the right end of their orthogonality interval and on the limiting properties of the $q$-Pochhammer and $q$-Gamma functions
 as $q\nearrow 1$.
 \end{abstract}
\maketitle

\section{Introduction and main results}\label{sec:Intro}
In this paper we study  Markov processes that arise as limits of  random Motzkin paths with random endpoints.
A  Motzkin path of length $L$ is a sequence of steps on the integer lattice that starts at point  $(0,n_0)$  with the initial altitude $n_0$ and ends at point $(L, n_L)$ at the final   altitude   $n_L$
for some non-negative integers $n_0,n_L,L$.
Steps can be upward, downward, or horizontal, along the vectors $(1,1)$, $(1,-1)$ and $(1,0)$, respectively,
and the path cannot fall below the horizontal axis.
The path is uniquely determined by the sequence of altitudes $\vv \gamma=(n_0,n_1,\dots,n_L)$   with $n_j-n_{j-1}\in\{-1,0,1\}$, $j=1,\dots,L$.

A random Motzkin path can be generated by assigning a discrete probability measure on the set of all Motzkin paths and choosing $\vv \gamma$ at random according to this probability measure.
We will write $\vv \gamma=(\gamma_0,\dots,\gamma_L)$ and will use the notation $\vv \gamma \topp L$ if we need to explicitly indicate its dependence on the parameter $L$.
In our main results, we are interested in  the assignment of  probability $\Pr_L(\vv \gamma)$ which depends on two boundary parameters $\rho_0,\rho_1\in[0,1)$ and two parameters $q\in[0,1)$ and $\sigma\in[0,1]$ that determine the edge weights. As the probability $\Pr_L(\vv \gamma)$ of selecting a Motzkin path $\vv\gamma$ we take the following expression:
\begin{equation}\label{Pr0}
   \Pr_L(\vv \gamma)=\frac{1}{\mathfrak C_L} \rho_0^{\gamma_0} \rho_1^{\gamma_L}(2\sigma)^{H(\vv \gamma)} \prod_{k=0}^{L}[\gamma_k+1]_q
\end{equation}
where   $[n]_q =1+\dots+q^{n-1}$ denotes the $q$-number, $H(\vv \gamma)=\#\{j\in[1,L]: \gamma_j=\gamma_{j-1}\}$ is the number of horizontal steps, and $\mathfrak C_L$ is the normalizing constant.  This weighting of the Motzkin paths was inspired by a formula
in \cite[Section 2.3]{barraquand2023Motzkin}, where $\sigma=1$.
A more general setup is discussed in Section \ref{sec:WMP}.

The above setup differs from the most commonly studied random Motzkin paths chosen uniformly from all Motzkin paths which start and end at altitude 0. Such random Motzkin paths correspond to  random walks conditioned on staying non-negative and  returning to $0$ at time $L$; it is well known that their asymptotic fluctuations  are described by the Brownian excursion \cite{kaigh1976invariance} and their %
behavior near boundaries
is described by an explicit Markov chain, see \cite{keener1992limit}. It turns out that new phenomena and new asymptotic fluctuations arise when the start and end of Motzkin paths are random as in \cite{Bryc-Wang-2023b,Bryc-Wang-2023}.
Our goal  is to extend \cite{Bryc-Wang-2023b} to allow more general  weights that depend on the parameter $q$. We then analyze the boundary limit Markov chain in two different asymptotic regimes.
Interestingly,  the limits in these regimes  recover  Markov processes that appeared in the description presented in \cite{Bryc-Kuznetsov-2021} of the
non-Gaussian term of the stationary measure of the KPZ equation on the half-line, see
\cite{barraquand2022steady}, \cite{barraquand2022stationary},   as well as the non-Gaussian
term in the conjectural stationary measure of the hypothetical KPZ fixed point on the half-line postulated in  \cite{barraquand2022steady}.

We use the following  standard notation  for the $q$-Pochhammer symbols:
\begin{align*}
   (a;q)_n&=\prod_{k=0}^{n-1} (1-a q^k),&
    (a_1,\dots,a_k;q)_n&=(a_1;q)_{n}(a_2;q)_{n}\dots (a_k;q)_n,\\
       (a;q)_\infty&=\prod_{k=0}^\infty (1-a q^k),&
    (a_1,\dots,a_k;q)_\infty&=(a_1;q)_{\infty}(a_2;q)_\infty\dots (a_k;q)_\infty.
\end{align*}
Our first main result is the following  limit theorem for the boundaries of the random Motzkin path. Let
$\vv\gamma\topp L=\{\gamma\topp L_k\}_{k\geq 0}$ be a sequence of the initial altitudes of a random Motzkin path of length $L$, appended with $0$ for $k>L$,
and let $\wt{\vv\gamma}\topp L=\{\wt {\gamma}\topp L_k\}_{k\in\ZZ_{\geq 0}}$ be a sequence of the final altitudes,  $\wt  \gamma_{k}\topp L =\gamma_{L-k}$, $k=0,1,\dots,L$, appended with $0$ for $k>L$.
\begin{theorem}\label{Thm:1} Suppose that $0\leq \rho_0,\rho_1,q<1$ and $0<\sigma \leq 1$.
   Then
   \[
\pp{\vv\gamma\topp L,\,\wt{\vv\gamma}\topp L}\weakto \pp{\vv X,\,\vv Y},
\]
 as $L\to\infty$, where on the left-hand   side we have random Motzkin paths with respect to measure \eqref{Pr0} and on the right-hand side we have two independent Markov chains $\vv X=\ccbb{X_k}_{k\ge 0}$, $\vv Y=\ccbb{Y_k}_{k\ge 0}$ with the same transition probabilities
 on $\ZZ_{\geq 0}$
given by
\begin{equation*}\label{tr*}
 \Pr(X_{k}=n+\delta|X_{k-1}=n)=\frac{1-q^{n+1}}{1+\sigma}\cdot   \begin{cases}\frac{s_{n+1}}{2s_n}, &\delta=1, \\
 \sigma, & \delta=0, \\
\frac{s_{n-1}}{2s_n}, & \delta=-1,\\
0, & \mbox{\em otherwise},
\end{cases}
\end{equation*}
 for $n=0,1,\ldots$  and with the initial laws
\begin{equation}
  \label{ini-0-laws}
  \Pr(X_0=n)=\frac{1}{C_{0}} \rho_0^n s_n,\;
\Pr(Y_0=n)=\frac{1}{C_{1}} \rho_1^n s_n,\quad  n=0,1,\ldots,
\end{equation}
where  $s_{-1}=0$ and
$$s_n =\sum_{k=0}^n \frac{(a;q)_k  }{
    (q;q)_k } \;\frac{ (b;q)_{n-k}}{
      (q;q)_{n-k} },\quad n=0,1,\ldots,
$$
with $a=-q(\sigma + \i\sqrt{1-\sigma^2})$, $b=-q(\sigma - \i\sqrt{1-\sigma^2})$;
 the normalizing constants are
 $$C_j=\frac{(a\rho_j,b\rho_j;q)_\infty}{(\rho_j;q)_\infty^2}, \quad j=0,1.$$
\end{theorem}

We note that if $q=0$ then $s_n=n+1$, recovering the transition probabilities in \cite[Theorem 1.1]{Bryc-Wang-2023b}, who use the parameter $\sigma$ that is twice our $\sigma$.   It is natural to expect that as in \cite[Theorem 1.6]{Bryc-Wang-2023b}, there is a version of this result that also holds for $\rho_1\geq 1$. However, this is beyond the scope of this paper.

Theorem \ref{Thm:1} will be deduced from a more general Theorem \ref{T-L}. The proof appears in Section \ref{sec:SCI} and relies on properties of the Al-Salam--Chihara polynomials, which we discuss in Section \ref{sec:AlSal}.

The next two theorems give macroscopic continuous-time limits of the family of Markov chains $\{X_k\}$.
 In the statements of Theorems \ref{Thm:2} and \ref{Thm:3} below,  $\fddto$ denotes convergence of finite-dimensional distributions.

In our first result, we take %
 $\rho_0\nearrow 1$
at an appropriate rate but keep $q$ fixed.
Then the normalized Markov chain $\ccbb{X_k}$   converges to the Bessel process,  which for $\sigma=1$ appeared in the description of the non-Gaussian %
term in
the conjectural stationary measure for the hypothetical KPZ fixed point on the half-line in \cite[Theorem 2.6 ]{Bryc-Kuznetsov-2021}.
\begin{theorem}\label{Thm:2}
  Fix $0\leq q<1$, $0<\sigma\leq 1$ and $\C>0$. Let $\{X_k\topp N\}_{k\in\ZZ_{\geq 0}}$ be a Markov chain from Theorem \ref{Thm:1} with the initial law that depends on $N$ through $\rho_0=e^{-\C/\sqrt{N}}$.
  Then
  \begin{equation}\label{12conv}
     \frac{1}{\sqrt{N}}\ccbb{X_{\floor{Nt}}\topp N}_{t\geq 0}\fddto \ccbb{\xi_t\topp \C}_{t\geq 0}\; \mbox{ as }   N\to\infty,
  \end{equation}
   where   $\{\xi_t\topp \C\}_{t\ge 0}$  is the 3-dimensional Bessel process with transition probabilities
   \begin{equation}\label{xi-tr}
       \Pr(\xi_t\topp \C=dy|\xi_0\topp \C=x)= \frac{y}{x}\,\g_{\frac{t}{1+\sigma}}(x,y)\, dy, %
   \end{equation}
   where $\g_t$, $t>0$, is the transition kernel of the Brownian motion killed at hitting zero, i.e.
   \begin{equation}\label{BMhit0}
   \g_t(x,y)=\frac{1}{\sqrt{2\pi t}}\left(\exp(-\tfrac{(y-x)^2}{2t})-\exp(-\tfrac{(y+x)^2}{2t})\right),\quad
   x,y>0,
   \end{equation}
and  with the  initial distribution
    \begin{equation}
      \label{xi0}
      \Pr( \xi_0\topp \C=d x)=\C^2 x e^{-\C x} {\bf 1}_{\{x>0\}} d x.
    \end{equation}
\end{theorem}
In our second result, we take both %
$\rho_0\nearrow 1$ and $q\nearrow 1$  at  appropriate rates.
Then, under appropriate centering and normalization, the Markov chain  $\ccbb{X_k}$ converges to the Markov process on $\RR$,   which for $\sigma=1$ appeared in the description of the non-Gaussian term in the representation of the stationary measure for the KPZ equation on the half-line given in \cite[Theorem 2.3]{Bryc-Kuznetsov-2021}. (This is only a one-parameter subset of the two-parameter family of such measures conjectured in \cite{barraquand2022steady} and proved rigorously in \cite{barraquand2022stationary}.) In the statement,
$K_\nu(z)$ is a modified Bessel $K$ function (Macdonald function) of imaginary index $\nu$ and positive argument $z$, see e.g. %
\cite[\href{https://dlmf.nist.gov/10.32.E9}{10.32.E9}]{NIST:DLMF}.
\begin{theorem}\label{Thm:3}
  Fix  $0<\sigma\leq 1$  and $\C>0$. Let $\{X_k\topp N\}_{k\in\ZZ_{\geq 0}}$ be a Markov process from Theorem \ref{Thm:1} with parameters $q=e^{-2/\sqrt{N}}$ and $\rho_0=e^{-\C/\sqrt{N}}$. Then
  \begin{equation}\label{13conc}
     \frac{1}{\sqrt{N}}\ccbb{X_{\floor{Nt}}\topp N-\sqrt{N}\log \sqrt{2 N(1+\sigma)}}_{t\geq 0}\fddto \ccbb{\zeta_t\topp \C}_{t\geq 0}\; \mbox{ as }  N\to\infty,
  \end{equation}
   where   $\zeta\topp \C$  is  a Markov process with transition probabilities
   \begin{equation}\label{zeta-tr}
       \Pr(\zeta_t^{\topp{\C}}=dy|\zeta_0^{\topp \C}=x)=    \frac{K_0(e^{-y})}{K_0(e^{-x})}\; \p_{\frac{t}{1+\sigma}}(x,y)\,d y,\quad
       x,y\in\RR,
   \end{equation}
   where  $\p_t$, $t>0$, is the Yakubovich transition kernel,  \cite{Yakubovich:2011}, defined by  \begin{equation}\label{Yak}\p_t(x,y)=\frac{2}{\pi}\int_0^\infty e^{-tu^2/2} K_{\i u}(e^{-x})K_{\i u}(e^{-y}) \frac{ du}{  |\Gamma(\i u)|^2},\end{equation}
and with the initial distribution
    \begin{equation}
      \label{zeta0}
      \Pr( \zeta_0\topp \C=d x)= \frac{4}{2^\C\Gamma(\C/2)^2} e^{-\C x} K_0(e^{-x})\,d x.
    \end{equation}
\end{theorem}

Theorems \ref{Thm:2} and \ref{Thm:3} will be derived from local limits in Theorem \ref{Thm-loc-lim} and Theorem \ref{Thm-q=1}.

The paper is organized as follows. In Section \ref{sec:WMP} we introduce random Motzkin paths with general weights, we prove the matrix representation, the integral representation, and the general boundary limit theorem.
In Section \ref{sec:AlSal} we recall some properties of the
Al--Salam--Chihara polynomials, and we prove  asymptotic results for their behavior near the right end of the interval of orthogonality; Theorem \ref{Thm:3.1} could be derived from known results but we believe Theorem \ref{Thm:3.2} is  new.  In Section \ref{sec:SCI} we  use these results to prove  Theorem \ref{Thm:1} and two local limit theorems, Theorem \ref{Thm-loc-lim} and Theorem \ref{Thm-q=1},   which we use to complete proofs of Theorems \ref{Thm:2} and \ref{Thm:3}.
 In Section \ref{Sec:SF} we discuss properties of the $q$-Pochhammer and $q$-Gamma functions  and derive the limits and bounds needed in our proofs.
\section{A general  limit theorem for random Motzkin paths at the boundary}\label{sec:WMP}
In this section we introduce a %
general setting and prove a version of Theorem \ref{Thm:1} for %
general random Motzkin paths.

Recall that a Motzkin path of length $L=1,2,\dots$  is a sequence
of lattice points $(\vv x_0,\dots, \vv x_L)$  such that $\vv x_j = (j, n_j)\in\ZZ_{\geq 0}\times  \ZZ_{\geq 0},\;j=0,1,\dots,L$. An edge $(\vv x_{j-1},\vv x_j)$  is
called a ($j$-th) step, and only three types of steps are allowed: upward steps, downward steps, and horizontal steps.
The edge $(\vv x_{j-1},\vv x_j)$ is an upward step if
$n_{j}-n_{j-1}=1$, a downward step if $n_{j}-n_{j-1}=-1$, and a horizontal step if $n_{j}-n_{j-1}=0$, see, e.g., \citet[Definition V.4, p. 319]{flajolet09analytic} or \cite{Viennot-1984a}.
Each such path can be identified with a  sequence of non-negative integers that specify the starting point $(0,n_0)$ and consecutive   values $n_j$  along the vertical axis at step $j\geq 1$.
We shall write  $\vv\gamma=(\gamma_0,\gamma_1,\dots,\gamma_L)$ with $\gamma_j = n_j$  for such a sequence and refer to $\vv\gamma$ as a Motzkin path.   By $\calM_{m,n}\topp L$ we denote the family of all Motzkin paths $\vv \gamma$  of length $L$ with the initial altitude $\gamma_0=m$ and the final altitude $\gamma_L=n$. We also refer to $\gamma_0$ and $\gamma_L$ as the boundary/endpoints of the path.

We are interested in the probabilistic properties of random Motzkin paths. To introduce a probability measure on the set of the Motzkin paths, we assign weights to the edges and to the endpoints of a Motzkin path. The weights for the edges arise multiplicatively from three sequences $\vv a=(a_j)_{j\geq 0},\vv b=(b_j)_{j\geq 0},\vv c=(c_j)_{j\geq 0}$ of positive real numbers.
For a path $\vv\gamma=(\gamma_0=m,\gamma_1,\dots,\gamma_{L-1},\gamma_L=n)\in\calM_{m,n}\topp L$ we define its (edge) weight
\begin{equation}\label{w}
  w(\vv\gamma)=\prod_{k=1}^L a_{\gamma_{k-1}}^{\eps_k^+} b_{\gamma_{k-1}}^{\eps_k^0}c_{\gamma_{k-1}}^{\eps_k^-},
\end{equation}
where
\begin{equation*}\label{epsilons}
   \eps_k^+(\vv\gamma):=\ind_{\gamma_{k}>\gamma_{k-1}},\; \eps_k^-(\vv\gamma):=\ind_{\gamma_{k}<\gamma_{k-1}}, \; \eps_k^0(\vv\gamma):=\ind_{\gamma_{k}=\gamma_{k-1}},\; k=1,\dots, L.
\end{equation*}
That is, the edge weight is multiplicative in the edges, we take $\vv a$, $\vv b$ and $\vv c$ as the weights of the upward steps, horizontal steps and downward steps, and the weight of a step depends  on the altitude of the left-end of an edge.
Note that the value of $c_0$ does not contribute to $w(\vv\gamma)$ and it is convenient to allow $c_0=0$. Since $\calM_{m,n}\topp L$ is a finite set, the normalization constants
\begin{equation*}\label{Wij}
  \mathfrak W
  \topp L_{m,n}=\sum_{\vv\gamma\in\calM_{m,n}\topp L} w(\vv\gamma)
\end{equation*}
are well defined for all $m,n\geq 0$.

  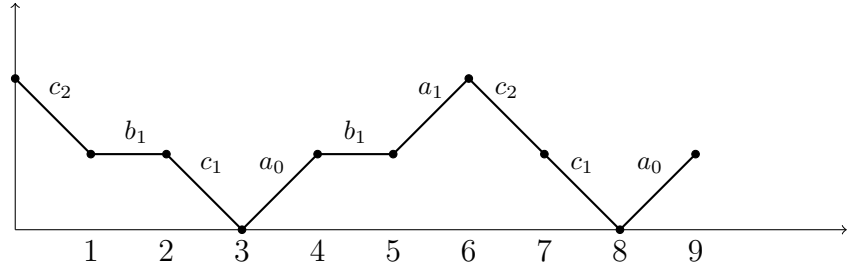
\begin{figure}[H]
  \begin{tikzpicture}[scale=1]

 \draw[->] (0,0) to (0,3);

 \draw[->] (0,0) to (11,0);
\draw[-,thick] (0,2) to (1,1);
\draw [fill] (0,2) circle [radius=0.05];
 \node[above] at (.6,1.6) {\footnotesize  $c_2$};
\draw[-,thick] (1,1) to (2,1);
\draw [fill] (1,1) circle [radius=0.05];
 \node[above] at (1.6,1) {\footnotesize  $b_1$};
\draw[-,thick] (2,1) to (3,0);
\draw [fill] (2,1) circle [radius=0.05];
 \node[above] at (2.6,0.6) {\footnotesize  $c_1$};
\draw[-,thick] (3,0) to (4,1);
\draw [fill] (3,0) circle [radius=0.05];
 \node[above] at (3.4,0.6) {\footnotesize  $a_0$};
\draw[-,thick] (4,1) to (5,1);
\draw [fill] (4,1) circle [radius=0.05];
 \node[above] at (4.5,1) {\footnotesize  $b_1$};
\draw[-,thick] (5,1) to (6,2);
\draw [fill] (5,1) circle [radius=0.05];
 \node[above] at (5.5,1.6) {\footnotesize  $a_1$};
\draw[-,thick] (6,2) to (7,1);
\draw [fill] (6,2) circle [radius=0.05];
 \node[above] at (6.5,1.6) {\footnotesize  $c_2$};
\draw[-,thick] (7,1) to (8,0);
\draw [fill] (7,1) circle [radius=0.05];
 \node[above] at (7.5,0.6) {\footnotesize  $c_1$};
\draw[-,thick] (8,0) to (9,1);
\draw [fill] (8,0) circle [radius=0.05];
 \node[above] at (8.4,0.6) {\footnotesize  $a_0$};
 \draw [fill] (9,1) circle [radius=0.05];

   \node[below] at (1,0) {  $1$};
      \node[below] at (2,0) {  $2$};
       \node[below] at (3,0) {  $3$};
        \node[below] at (4,0) {  $4$};
         \node[below] at (5,0) {  $5$};
          \node[below] at (6,0) {  $6$};
           \node[below] at (7,0) {  $7$};
            \node[below] at (8,0) {  $8$};
             \node[below] at (9,0) {  $9$};

\end{tikzpicture}

\caption{\label{Fig1}    Motzkin path $\vv\gamma=(2,1,1,0,1,1,2,1,0,1)\in\calM\topp 9$  with weight contributions marked at the edges. The probability of selecting the path shown above from $\calM\topp 9$ is
$\Pr_9(\vv\gamma)=\frac{\alpha_2 \beta_1}{\mathfrak C_{\vv \alpha,\vv \beta, 9} }  b_1^2 a_0^2 a_1 c_1^2  c_2^2$.
}
\end{figure}

In addition to the weights of the edges, we also assign weights to the two boundary points.
To this end we choose two additional non-negative sequences $\vv \alpha=(\alpha_m)_{m\ge 0}$ and $\vv \beta=(\beta_n)_{n\ge 0}$ such that
\begin{equation}
   \label{CL} \mathfrak{C}_{\vv\alpha,\vv\beta,L}:=\sum_{m,n\geq 0} \alpha_m \mathfrak W\topp L_{m,n}\beta_n<\infty.
\end{equation}

With finite normalizing constant \eqref{CL}, the countable set  $\calM\topp L=\bigcup_{m,n\geq 0} \calM_{m,n}\topp L$ becomes a probability space with the discrete probability measure
\begin{equation}\label{Pr}
\Pr_L(\vv\gamma)\equiv  \Pr_{\vvalpha,\vvbeta,L}(\{\vv\gamma\})=\frac{\alpha_{\gamma_0}\beta_{\gamma_L}}{\mathfrak{C}_{\vvalpha,\vvbeta,L}}\, w(\vv\gamma), \quad \vv\gamma\in\calM\topp L.
\end{equation}
For an illustration, see Figure \ref{Fig1}.
We now consider a random vector $(\gamma_0\topp L,\dots,\gamma_L\topp L)\in\calM\topp L$ sampled from $\Pr_L$, and extend it to an infinite random sequence  denoted by
\begin{equation}\label{gam}
\vv\gamma\topp L:= \ccbb{\gamma_k\topp L}_{k\ge 0}, \qmwith \gamma_k\topp L := 0 \mbox{ if } k>L.
\end{equation}
Similarly, we also introduce a %
reversed random sequence %
\begin{equation}\label{wtgam}
\wt{\vv\gamma}\topp L:=\ccbb{\wt\gamma_k\topp L}_{k\ge 0} \qmwith \wt\gamma_k\topp L :=\begin{cases}
\gamma_{L-k}\topp L, & \mbox{ if } k=0,\dots,L,\\
0, & \mbox{ if } k>L.
\end{cases}
\end{equation}
  We are interested in %
 weak limits of  laws of
  $\vv\gamma\topp L$ and $\wt{\vv\gamma}\topp L$ as $L\to\infty$.
In these limits,   we do not scale the sequences nor the locations.

Recall that weak convergence of discrete-time processes means the convergence of finite-dimensional distributions \citep{billingsley99convergence}. For integer-valued random variables, the latter
follows from the convergence of probability generating functions. We will therefore   fix $K$, $z_0,z_1\in(0,1]$ and $t_1,\dots,t_K, s_1,\dots,s_K>0$ and the goal is to determine two discrete-time processes $\vv X=\{X_k\}_{k\ge 0}$ and $\vv Y=\{Y_k\}_{k\ge 0}$ such that
\begin{multline}\label{FixK}
\lim_{L\to\infty}  \EE \bb{z_0^{\gamma_0\topp L} \pp{\prod_{j=1}^K t_j^{\gamma_j\topp L-\gamma_{j-1}\topp L}}\pp{ \prod_{j=1}^K s_j^{\gamma_{L-j}\topp L-\gamma_{L+1-j}\topp L}}z_1^{\gamma_L\topp L}} \\
=\EE\bb{ z_0^{X_0} \prod_{j=1}^K t_j^{X_j-X_{j-1}}}\EE\bb{ z_1^{Y_0} \prod_{j=1}^K s_j^{Y_j-Y_{j-1}}}.
\end{multline}
Indeed, the above expressions uniquely determine the corresponding probability generating functions for small enough arguments. For example,
$$\EE\bb{ \prod_{j=0}^K v_j^{X_j}}=\EE\bb{ z_0^{X_0} \prod_{j=1}^K t_j^{X_j-X_{j-1}}}$$
with
$z_0=v_0\dots v_K$ and $t_j=v_jv_{j+1}\dots v_K$.
\subsection{Matrix representation and integral representation} We will need a convenient representation for
the expectation on the left-hand side of \eqref{FixK}.
We introduce a tri-diagonal matrix
\begin{equation*}\label{MM}\mM_t:=\left[\begin{matrix}
  b_0 & a_0 t & 0 &0 &\cdots \\
  \tfrac{c_1}{t} & b_1 &a_1 t &0 &\\
  0 & \tfrac{c_2}{t} & b_2&a_2 t &\\
 0 & 0& \tfrac{c_3}{t} &  b_3 &\ddots\\
 \vdots &&&\ddots&\ddots
\end{matrix}\right],
\end{equation*}
and infinite column vectors
$$
\vec V_\vvalpha(z) :=[\alpha_n\,z^n,\,n=0,1,\ldots]^T\quad\mbox{and}\quad\vec W_\vvbeta(z) :=[\beta_n\,z^n,\,n=0,1,\ldots]^T.
$$
We use the following matrix representation for the left-hand side of \eqref{FixK}.
\begin{lemma}[matrix ansatz] \label{L-mat-ans}
For every $K=0,1,\dots$ such that $2K\le L$,  we have
\begin{multline}\label{*}
\EE \bb{z_0^{\gamma_0\topp L} \pp{\prod_{j=1}^K t_j^{\gamma_j\topp L-\gamma_{j-1}\topp L}}\pp{ \prod_{j=1}^K s_j^{\gamma_{L-j}\topp L-\gamma_{L+1-j}\topp L}}z_1^{\gamma_L\topp L}}
\\=\frac{1}{\mathfrak C_{\vvalpha,\vvbeta,L}} \Vec V_\vvalpha(z_0)^T \mM_{t_1}\mM_{t_2}
\cdots \mM_{t_K} \mM_1^{L-2K} \mM_{1/s_K}\cdots \mM_{1/s_2}\mM_{1/s_1}\vec W_\vvbeta(z_1),
\end{multline}
where
\begin{equation}\label{cfrak}\mathfrak C_{\vvalpha,\vvbeta,L}=\vec V_{\vv \alpha}(1)^T \mM_1^L\vec W_{\vv \beta}(1)\end{equation} is the normalization constant.
\end{lemma}

\begin{proof}
Formula \eqref{*} follows from a more general formula
\begin{equation}\label{m-ans}
 \sum_{m,n=0}^\infty \alpha_m z_0^m \beta_n z_1^n \sum_{\vv \gamma\in \calM_{m,n}^{\topp L}}w(\vv \gamma)\prod_{j=1}^L t_j^{\gamma_j-\gamma_{j-1}} =
 \Vec V_\vvalpha(z_0)^T \mM_{t_1}\mM_{t_2}
\cdots \mM_{t_L} \vec W_\vvbeta(z_1)
\end{equation}
applied to the sequence $(t_1,\dots,t_L)$ of the form $(t_1,\dots,t_K,1,\dots,1,1/s_K,\dots,1/s_1)$.

To prove \eqref{m-ans}, consider an infinite matrix  $\mW(L)$ with entries
\begin{equation}\label{mwL}[\mW (L)]_{m,n}:=\sum_{\vv\gamma\in \calM_{m,n}^{\topp L}}w(\vv \gamma)\prod_{j=1}^L t_j^{\gamma_j-\gamma_{j-1}} .\end{equation}
It is clear that the left-hand side of \eqref{m-ans} is
$$\Vec V_\vvalpha(z_0)^T  \mW(L) \vec W_\vvbeta(z_1).$$
It remains to  verify that
$$\mW(L)=
\mM_{t_1}\mM_{t_2}
\cdots \mM_{t_L}.$$
With the case $L=0$ being trivial, we proceed by induction.  In view of \eqref{mwL} we have
$$[\mW(L+1)]_{m,n}=t_{L+1}[\mW(L)]_{m,n-1}a_{n-1}+[\mW(L)]_{m,n}b_n+t_{L+1}^{-1} [\mW(L)]_{m,n+1}c_{n+1}.$$
Thus $\mW(L+1)=\mW(L)\mM_{t_{L+1}}$.
\end{proof}

The key step in our proof of the boundary limit theorem  is an integral representation for the left-hand side of \eqref{FixK} based on the orthogonality measure of the orthogonal polynomials determined by the edge weights of Motzkin paths.

 Following
\cite{Viennot-1984a} and \cite{flajolet09analytic} with
the sequences  $\vv a,\vv b, \vv c$ of edge weights for the Motzkin paths that appeared in \eqref{w} we now
associate real polynomials  $p_{-1}(x)=0$, $p_0(x)=1,p_1(x),\dots$ defined  by the three step recurrence relation
\begin{equation}
  \label{3-step}
  x p_n(x)=a_n p_{n+1}(x)+b_n p_n(x)+c_np_{n-1}(x),\quad n=0,1,2,\dots
\end{equation}
(With the above conventions that $p_{-1}\equiv 0$,  $p_0\equiv 1$ and $a_n>0$, the recurrence relation determines polynomials $\{p_n(x)\}$ uniquely, with $p_1(x)=(x-b_0)/a_0$.)
By Favard's theorem \citep{ismail09classical}, polynomials $\{p_n(x)\}$  are orthogonal with respect to a probability measure $\nu$ on the real line. We assume that $\vv  a, \vv b,\vv c$ are such that $\nu$ is compactly supported and thus unique.
It is well known that the $L_2$ norm   $ \|p_n\|_2^2:=\int_\RR (p_n(x))^2 \nu(dx)$ is given by the formula
\begin{equation}\label{L2}
  \|p_n\|_2^2=\prod_{k=1}^n \frac{c_k}{a_{k-1}}.
\end{equation}
We need some conditions on the weights of the endpoints $\vvalpha,\vvbeta$ in \eqref{Pr}.  %
\begin{assumption} \label{A-A1A2}
 We assume that
\begin{enumerate}
    \item[($A_1$)]   The series
\begin{equation}\label{Phi}
\phi_{\vv \alpha}(x,z)= \sum_{n=0}^\infty \alpha_n z^n p_n(x)\quad\mbox{ and }\quad
\psi_{\vv \beta}(x,z)=\sum_{n=0}^\infty \beta_n z^n
\frac{p_n(x)}{\|p_n\|_2^2}
\end{equation}
converge absolutely on the support of $\nu$ for all $z\in\CC, |z|\leq 1$.
\item[($A_2$)]
\begin{equation}
  \label{A-2} %
 \int_\RR \sum_{m,n=0}^\infty \alpha_n  \frac{\beta_m} {\|p_m\|_2^2} \abs{x^Lp_n(x)p_m(x)}\nu(dx)<\infty.
\end{equation}
\end{enumerate}
\end{assumption}
In particular, the function $x\mapsto x^L\phi_\vvalpha(x,1)\psi_\vvbeta(x,1)$ is integrable with respect to the measure $\nu$.

Consider the following two infinite column vectors
$$
\vec P(x):=\left[p_i(x),\,i=0,1,\ldots\right]^T\quad \mbox{and}\quad \vec Q(x):=[\wt p_i(x),\,i=0,1,\ldots]^T,
$$
where
\begin{equation}\label{wtp}
\wt p_n(x)=\frac{p_n(x)}{\|p_n\|_2^2}=p_n(x) \prod_{k=1}^{n}\frac{a_{k-1}}{c_k}.
\end{equation}

 Note that  with $a_{-1}:=0$, polynomials $\wt p_n(x)$ %
 satisfy the three step recurrence relation
 $$
 x \wt p_n(x)
 =c_{n+1}\wt p_{n+1}(x)+b_n\wt p_n(x)+ a_{n-1} \wt p_{n-1}(x).
 $$
 In particular,
 \begin{equation}\label{mpq}
 %\mM_1^L \vec P(x)=x^L \vec P(x)\quad \mand \quad  \vec Q^T(x) \mM_1^L=x^L\vec Q^T .
 \mM_1^j \vec P(x)=x^j \vec P(x),\quad j\ge 0.
  \end{equation}
The left-hand side of the expression \eqref{*} has the following integral representation.
 \begin{lemma}
   \label{L-IF}
   If \eqref{A-2} holds  and $2K\le L$, then
   \begin{multline}\label{*I}
\EE \bb{z_0^{\gamma_0\topp L} \pp{\prod_{j=1}^K t_j^{\gamma_j\topp L-\gamma_{j-1}\topp L} }\pp{\prod_{j=1}^K s_j^{\gamma_{L-j}\topp L-\gamma_{L+1-j}\topp L}}z_1^{\gamma_L\topp L}}
 =\frac{1}{\mathfrak C_{\vvalpha,\vvbeta,L}} \int x^{L-2K} \Psi_0(x)\Psi_1(x)
 \nu(dx),
\end{multline}
where
$$
\Psi_0(x)={\Vec V_\vvalpha(z_0)^T \mM_{t_1}\mM_{t_2} \cdots \mM_{t_K}  \vec P(x)} \mbox{ and } \Psi_1(x)= {\vec W_\vvbeta(z_1)^T \wt \mM_{s_1}\wt \mM_{s_2} \cdots \wt\mM_{s_K}\vec Q(x)},
$$
with $\wt \mM_{s}=\mM_{1/s}^T$. Moreover, $\mathfrak C_{\vvalpha,\vvbeta,L}$, introduced in \eqref{cfrak}, can be written as
$$\mathfrak C_{\vvalpha,\vvbeta,L}=\int x^L\pp{\vec V_{\vv \alpha}(1)^T \vec P(x)}
\pp{\vec W_{\vv \beta}(1)^T \vec Q(x)}\nu(dx).$$
 \end{lemma}
 \begin{proof}
 We use    expression \eqref{m-ans} with $(t_1,\dots,t_L)$ of the form $$(t_1,\dots,t_K,1,\dots,1,1/s_K,\dots,1/s_1).$$
 Proceeding somewhat formally, we write the infinite  identity matrix as the integral,
  $$\mI=\int\, \vec P(x) \vec Q(x)^{T}\,\nu(dx)$$ and rewrite the product of matrices in \eqref{*} as
 \begin{align*}
  &\mM_{t_1}\mM_{t_2}
\cdots \mM_{t_K} \mM_1^{L-2K} \mM_{1/s_K}\cdots \mM_{1/s_2}\mM_{1/s_1}
\\=& \mM_{t_1}\mM_{t_2}
\cdots \mM_{t_K} \mM_1^{L-2K}\pp{\int\,\vec P(x)\vec Q(x)^T\,\nu(dx)}\, \mM_{1/s_K}\cdots \mM_{1/s_2}\mM_{1/s_1}
\\=&
\int \mM_{t_1}\mM_{t_2}
\cdots \mM_{t_K}\pp{\mM_1^{L-2K}\vec P(x)}\,\vec Q(x)^T\, \mM_{1/s_K}\cdots \mM_{1/s_2}\mM_{1/s_1}\nu(dx)
\\\stackrel{\eqref{mpq}}{=}&\int x^{L-2K} \mM_{t_1}\mM_{t_2}
\cdots \mM_{t_K}\vec P(x) \,\vec Q(x)^T \mM_{1/s_K}\cdots \mM_{1/s_2}\mM_{1/s_1}\nu(dx).
 \end{align*}
 Therefore, the right-hand side of \eqref{*} becomes  %
\begin{multline*}
   \tfrac{1}{\mathfrak C_{\vvalpha,\vvbeta,L}} \, \int x^{L-2K} \pp{\Vec V_\vvalpha(z_0)^T \mM_{t_1} \mM_{t_2}
\cdots \mM_{t_K}\vec P(x)} %
\pp{ \vec Q(x)^T \mM_{\frac{1}{s_K}}\cdots \mM_{\frac{1}{s_2}}\mM_{\frac{1}{s_1}}\vec W_\vvbeta(z_1)} \nu(dx),
\end{multline*}
 and the result follows from the symmetry of the dot product, $\vec u^T\vec v=\vec v^T \vec u$.
 (Assumption \eqref{A-2} is responsible for Fubini's theorem, which is used to switch the integral with the infinite sums above. This avoids formal integration of the product of infinite matrices.)
 \end{proof}
 We will need the following elementary lemma. (See also \cite[Lemma A.1]{Bryc-Wang-2023b}.)
 \begin{lemma}
  \label{L-lim}
  Let  $\mu$ be a probability measure with  $B\in {\textnormal{supp}}(\mu) \subseteq [A,B]$ with $|A|<B<\infty$. Let $F$ be a real function which is bounded on $\mathrm{supp}(\mu)$ and left-continuous at $B$.

  Then
  \begin{equation}
    \label{Alexey}
    \lim_{L\to\infty}\frac{\int F(x)x^L \mu(dx)}{\int x^L \mu(dx)}=F(B).
  \end{equation}
\end{lemma}

\begin{proof}%
   Let $X$ be a random variable with the law $\mu$. First, we will derive some estimates on the moments of $X$ and $|X|$. Fix $\eps>0$ such that $B-\eps>|A|$.
   Since $B \in {\textnormal{supp}}(\mu)$,
   we have $C:=\Pr(X>B-\eps/2)>0$.
   And since ${\textnormal{supp}}(\mu) \subseteq [A,B]$ we have
   \begin{multline*}\EE [X^L]=\EE[X^L \ind_{\{X>  B-\eps/2\}}] +\EE[X^L \ind_{\{X \leq B-\eps/2\}}]
   \geq C (B-\eps/2)^L - |A|^L
  \\ =(B-\eps/2)^L \pp{C-\frac{|A|^L}{(B-\eps/2)^L}}.
   \end{multline*}
  The inequalities $B-\epsilon/2>B-\epsilon>|A|$ imply that the term $|A|^L/(B-\epsilon/2)^L$ converges to zero as $L\to +\infty$, therefore $\EE [X^L]> \frac{1}{2} C(B-\eps/2)^L$ for large enough $L\in {\mathbb N}$.
   Next, we use the inequality $|x|^L\le x^L + 2|A|^L$, which is valid for all $x\in [A,B]$ and $L\in {\mathbb N}$, and obtain
   $$
   \frac{\EE[|X|^L]}{\EE[X^L]}\le 1 + 2 \frac{|A|^L}{\EE[X^L]}<1+\frac{4}{C} \frac{|A|^L}{(B-\epsilon/2)^L},
   $$
   which gives us an upper bound $\EE[|X|^L]<2\EE[X^L]$, valid for $L$ large enough.

  Now we are ready to establish \eqref{Alexey}.
 By taking $F-F(B)$ instead of $F$ in \eqref{Alexey}, without loss of generality, we can assume that $F(B)=0$ (clearly,  $F-F(B)$ remains left-continuous at $B$).  Now we estimate, for $L\in {\mathbb N}$ large enough
   \begin{align*}
    \left | \frac{\EE[F(X)X^L ]}{\EE[X^L ]}\right|&\leq
    \frac{\EE[|F(X)||X|^L\ind_{\{X\leq B-\eps\}} ]}{\EE[X^L ]}+\frac{\EE[|F(X)||X|^L\ind_{\{X> B-\eps\}} ]}{\EE[X^L ]}
   \\ &\leq  \frac{2}{C} \|F\|_\infty  \frac{(B-\eps)^L}{(B-\eps/2)^L}+
    2\sup_{x\in[B-\eps,B]}|F(x)|,
   \end{align*}
  where we used inequalities
   $$
   \EE[|F(X)||X|^L\ind_{\{X\leq B-\eps\}} ]\leq \|F\|_\infty(B-\eps)^L \quad \mbox{and} \quad \EE[X^L]\geq \tfrac{1}{2} C(B-\tfrac{\eps}2)^L
   $$
   for the first term and
   $$
   \EE[|F(X)||X|^L\ind_{\{X> B-\eps\}} ]\leq\sup_{x\in[B-\eps,B]}|F(x)|\times \EE[|X|^L]\quad \mbox{and}\quad \EE[|X|^L]<2\EE[X^L]
   $$ for the second term.
   Since $(B-\epsilon)^L/(B-\epsilon/2)^L$ converges to zero when $L\to +\infty$, we obtain
   $$
  \limsup\limits_{L\to +\infty} \left | \frac{\EE[F(X)X^L ]}{\EE[X^L ]}\right| \leq  2 \sup_{x\in[B-\eps,B]}|F(x)|.
   $$
    Taking the limit $\epsilon \to 0^+$, gives us the desired result \eqref{Alexey} in the case $F(B)=0$.
\end{proof}

 \subsection{Limit theorem}

 The following result is a version of \cite[Theorem 1.1]{Bryc-Wang-2023b} for non-constant weights of edges.  It describes the joint limiting behavior of processes $\vv\gamma\topp L$ and $\wt{\vv\gamma}\topp L$ introduced in \eqref{gam} and \eqref{wtgam}.
 \begin{theorem}
   \label{T-L}
   Suppose that  $\textnormal{supp}(\nu)$,  the support of the orthogonality measure $\nu$ of the polynomials \eqref{3-step}, satisfies  $B\in {\textnormal{supp}}(\nu) \subseteq [A,B]$   with  $B>|A|$  and that
 $(A_1)$ and $(A_2)$ from    Assumption \ref{A-A1A2}   hold.
  Then $\pi_n:=p_n(B)>0$, $n=0,1,\ldots$  and
   \[
 \pp{\vv\gamma\topp L ,\,\wt{\vv\gamma}\topp L}\weakto \pp{\vv X,\vv Y}
\]
 as $L\to\infty$, where  $\vv X=\ccbb{X_k}_{k\ge 0}$ and $\vv Y=\ccbb{Y_k}_{k\ge 0}$ are   independent  Markov chains with the same transition probabilities
\begin{equation}\label{trans-geX}   \mathsf Q_{n,m}=\frac{1}{B}\begin{cases}
  a_n \frac{\pi_{n+1}}{\pi_n},   & m=n+1,\\
  b_n,   &m=n, \\
  c_n \frac{\pi_{n-1}}{\pi_n},  &m=n-1, \\
  0, & \mbox{otherwise},
 \end{cases}
\end{equation}
with  $m,n=0,1,\ldots$, and with the initial laws given by
\begin{equation}\label{ini-law}
  \Pr(X_0=n)=\frac{1}{C_{\vv\alpha}}\alpha_n \pi_n, \quad \Pr(Y_0=n)=\frac{1}{C_{\vv\beta}}\beta_n \wt \pi_n, \; n=0,1,\ldots,
\end{equation}
where $C_{\vv\alpha}$, $C_{\vv\beta}$ are the normalizing constants and $\wt \pi_n=\wt p_n(B)$ with $\wt p_n$ given in \eqref{wtp}.
 \end{theorem}
 \begin{proof}
We first verify that the transition probabilities \eqref{trans-geX} and \eqref{trans-geY} are well defined, i.e. $\pi_n>0$, $n=0,1,\dots$. Clearly, $\pi_0=1$ and recall that $\pi_{-1}=0$.
 For $n\geq 1$, the polynomial $p_n$ has a leading term $a_0^{-1} a_1^{-1} \dots a_{n-1}^{-1} x^n$, thus $p_n(x) \to +\infty$ as $x\to +\infty$.
  If $\pi_n=p_n(B)$ were negative, this would imply that $p_n$ has a zero outside of the support of the measure of orthogonality, which is impossible, see e.g. \cite{ismail09classical}. The case $p_n(B)=0$ is also impossible, since the interlacing property   of the zeroes of   orthogonal polynomials would imply that $p_{n+1}$ has a zero in the interval $(B,\infty)$.

To determine the limit as $L\to\infty$ of the left-hand side of \eqref{FixK},   we re-write the right-hand side  of \eqref{*I} as
$$
 \frac{\int x^L \nu(dx)}{\int x^L\pp{\vec V_{\vv \alpha}(1)^T \vec P(x)}
\pp{\vec W_{\vv \beta}(1)^T \vec Q(x)}\nu(dx)}  \cdot
 \frac{\int x^{L-2K} \nu(dx)}  {\int x^L \nu(dx)}
 \cdot \frac{\int x^{L-2K} \Psi_0(x) \Psi_1(x)
\nu(dx)}{\int x^{L-2K} \nu(dx)}.
$$
Using Lemma \ref{L-lim}, we see that each of the factors above converges as $L\to\infty$ and we get
\begin{multline} \label{Fact}
    \lim_{L\to\infty}  \EE \bb{z_0^{\gamma_0\topp L} \pp{\prod_{j=1}^K t_j^{\gamma_j\topp L-\gamma_{j-1}\topp L} }\pp{\prod_{j=1}^K s_j^{\gamma_{L-j}\topp L-\gamma_{L+1-j}\topp L}}z_1^{\gamma_L\topp L}}
    \\= \frac{1}{ \pp{\Vec V_\vvalpha(1)^T \vec P(B)}\pp{\vec Q(B)^T\vec W_\vvbeta(1)   }}
    \cdot \frac{1}{B^{2K}}\cdot \Psi_0(B) \Psi_1(B) = \frac{\Psi_0(B)}{B^{K}C_{\vv\alpha}}
  \cdot  \frac{\Psi_1(B)}{B^{K}C_{\vv\beta} }.
\end{multline}
We will show that the last %
expressions matches the product on the right-hand side of  \eqref{FixK}.

To this end it suffices to prove
\begin{equation}\label{mtmt}
\frac{1}{B^K}\,\Vec V_\vvalpha(z)^T \mM_{t_1}\mM_{t_2}  \cdots \mM_{t_K}\vec P(B)=\sum_{n\ge 0} \alpha_n\pi_n z^n\,\EE(T_{1:K}|X_0=n),
\end{equation}
where $T_{i:K}=\prod_{j=i}^K t_j^{X_j-X_{j-1}}$, as well as
\begin{equation}\label{msms}\frac{1}{B^K}\,\Vec W_\vvbeta(z)^T \wt \mM_{s_1}\wt \mM_{s_2} \cdots \wt\mM_{s_K}\vec Q(B)=\sum_{n\ge 0} \beta_n\wt \pi_n z^n\,\EE(S_{1:K}|Y_0=n),
\end{equation}
where $S_{i:K}=\prod_{j=i}^K s_j^{Y_j-Y_{j-1}}$,
and
\begin{equation}\label{trans-geY}
   \Pr(Y_{k+1}=n+\delta|Y_k=n)=\frac{1}{B}\begin{cases}
  c_{n+1} \frac{\wt \pi_{n+1}}{\wt \pi_n},  & \delta=1,\\
 b_n,   &\delta=0, \\
  a_{n-1} \frac{\wt \pi_{n-1}}{\wt \pi_n},  &\delta=-1.
 \end{cases}
\end{equation}
Note that processes $\vv X$ and $\vv Y$ have the same transition probabilities, since
$$
\frac{\wt \pi_{n+1}}{\wt \pi_n}=\frac{\pi_{n+1}}{\pi_n} \frac{a_n}{c_{n+1}}.
$$

We prove only \eqref{mtmt} since the proof of \eqref{msms} follows along the same lines. We use induction with respect to $K$.

The case of $K=0$ is immediate since the left-hand side of \eqref{mtmt} is $\Vec V_\vvalpha(z)^T\vec P(B)=\sum_{n\ge 0}\,\alpha_n\pi_n\,z^n$. For $K>0$ we first note that the left-hand side of \eqref{mtmt} can be written as
$$
R_K:=\frac{1}{B^{K-1}}\,\Vec V_{\wt \alpha(z,t_1)}(z)^T \mM_{t_2}  \cdots \mM_{t_K}\vec P(B),
$$
where $\wt \alpha_n(z,t)=(\alpha_{n-1}a_{n-1}\frac{t}{z}+\alpha_n b_n+\alpha_{n+1} c_{n+1}\frac{z}{t})/B$, $n\ge 0$. Consequently, by induction assumption for $K-1$ after shifting the indexes of summation we get
\begin{multline*}
R_K=\sum_{n\ge 0}\,\wt \alpha_n(z,t_1)z^n\pi_n\,\EE(T_{2:K}|X_1=n) %
=\sum_{n\ge 0}\Big(\alpha_{n-1}a_{n-1}\frac{t_1}{z}+\alpha_n b_n+\alpha_{n+1} c_{n+1}\frac{z}{t_1}\Big)\frac{\pi_nz^n}{B}\,\EE(T_{2:K}|X_1=n)\\
=\sum_{n\ge 0}\,\alpha_n\pi_nz^n\Big(t_1\EE(T_{2:K}|X_1=n+1)\,\frac{a_n\pi_{n+1}}{B\pi_n}
 +\EE(T_{2:K}|X_1=n)\,\frac{b_n}{B}
+\frac{1}{t_1}\,\EE(T_{2:K}|X_1=n-1)\,\frac{c_n\pi_{n-1}}{B\pi_n}\Big)\\
=\sum_{n\ge 0}\alpha_n\pi_n z^n\,\EE\left(t_1^{X_1-n}T_{2:K}|X_0=n\right)=\sum_{n\ge 0} \alpha_n\pi_n z^n\,\EE(T_{1:K}|X_0=n).
\end{multline*}

Formula \eqref{Fact}   shows that the limit processes $\vv X$ and $\vv Y$ are independent. This ends the proof.
\end{proof}

For proofs of local limit theorems, we need
to determine the limit of
$$\sqrt{N}\Pr\pp{X_{\floor{Nt}}=y_N \middle| X_{0}=x_N}$$ as $N\to\infty$
for suitably chosen sequences $(y_N)$ and $(x_N)$.
The following  formula %
 is
useful in computing  such limits.
\begin{proposition}
   For  the  Markov   chain   $\{X_k\}$, we have
      \begin{equation}
       \label{IR-2}
          \Pr(X_k=n|X_0=m)= \frac{\pi_n}{\pi_m} \frac{1}{B^k}\int_{[A,B]} x^k p_m(x)\wt p_n(x)\,\nu(d x).
   \end{equation}
\end{proposition}
(This resembles   \cite[p. 67]{karlin1959random}.)
\begin{proof} Recall notation \eqref{w}.
Fix $\vv \gamma\in \calM_{m,n}\topp k$. Due to cancellations,
$$\Pr(X_1=\gamma_1,X_2=\gamma_2,\dots,X_k=\gamma_k|X_0=\gamma_0)=\frac{w(\vv \gamma)}{B^k} \frac{\pi_{\gamma_k}}{\pi_{\gamma_0}},$$
so
$$\Pr(X_k=n|X_0=m)= \frac{\pi_n}{B^k\pi_m} \sum_{\vv \gamma \in\calM_{m,n}\topp k} w(\vv \gamma).  $$
By \citet[(5)]{Viennot-1984a} (see also \cite[Proposition 2.1]{Bryc-Wang-2023b}) we
have
\begin{equation*}  \label{Viennot5}
  \int_\RR p_m(x)p_n(x) x^L\,\nu(d x) =\|p_n\|_2^2\sum_{\gamma\in \calM_{m,n}\topp L} w(\vv \gamma).
\end{equation*}
Consequently,
$$\Pr(X_k=n|X_0=m)= \frac{\pi_n}{B^k\pi_m \|p_n\|_2^2}\int_{\RR}\,x^k p_m(x)p_n(x)\, \nu(d x).$$
\end{proof}
\section{Properties of the Al-Salam--Chihara polynomials}\label{sec:AlSal}
The general theory developed in Section \ref{sec:WMP} can be advanced further when it is specialized to the setting considered in Section \ref{sec:Intro}. To proceed, we need to recall the definition and  properties of  the Al-Salam--Chihara polynomials.

We will be working with the Al-Salam-Chihara polynomials $\{Q_n\}$ in real variable $x$, defined by the three term recurrence relation
\begin{multline}\label{Al-Salam}
2 x  Q_n(x;a,b|q)  =Q_{n+1}(x;a,b|q)+(a+b)q^n  Q_n(x;a,b|q)+(1-q^n)(1-ab q^{n-1}) Q_{n-1}(x;a,b|q),
\end{multline}
$n=0,1,\dots$,
where parameters  $a,b$ are either real or complex conjugate, $|ab|<1$, and $q\in[0,1)$. As usual, we initialize the recurrence with $Q_{-1}\equiv 0$ and $Q_{0}\equiv 1$.

It is known, see \cite{ismail09classical,koekoek2010hypergeometric} or \cite{koekoek1994askey}, that if
$|a|<1, |b|<1$, then the polynomials $\{Q_n\}$ are orthogonal with respect to the unique probability density supported on $[-1,1]$ and given by
\begin{equation}\label{Q-w}
g(x) = \frac{(q,ab;q)_\infty}{2\pi\sqrt{1-x^2}} \frac{\left|(e^{2 \i \theta};q)_\infty\right|^2}{\left|(a e^{\i\theta},be^{\i \theta};q)_\infty\right|^2}, \quad x=\cos \theta.
\end{equation}
It is also known that the generating function of $\{Q_n\}$  is well defined for complex $|t|<1$, and is given by the ratio of
infinite product:
\begin{equation}\label{Q-gen}
\sum_{n=0}^\infty \frac{Q_n(\cos \theta;a,b|q) }{(q;q)_n}t^n=\frac{(at,bt;q)_\infty}{(e^{\i\theta} t,e^{-\i\theta} t;q)_\infty}.
\end{equation}
Invoking the $q$-binomial theorem, see, e.g., \cite[(1.3.2)]{gasper2004basic}, the latter implies   the following formula, which can also be recalculated from \cite[ p.  50]{Berg-Ismail-96} or \cite[(15.1.12)]{ismail09classical}:
\begin{equation}\label{Berg-Ismail}
  \frac{Q_n(\cos \theta;a,b|q)}{(q;q)_n}= \sum_{k=0}^n \frac{(b e^{\i\theta};q)_k}{(q;q)_k}e^{-\i k \theta} \frac{(a e^{-\i\theta};q)_{n-k}}{(q;q)_{n-k}}e^{\i (n-k) \theta}.
\end{equation}
In particular, at the boundary of the interval of orthogonality, we have
\begin{equation}\label{Qn-expl}
   \frac{Q_n(1;a,b|q) }{(q;q)_n}=   \sum_{k=0}^n  \frac{ (a;q)_k (b;q)_{n-k}}{(q;q)_k(q;q)_{n-k}}.
\end{equation}

\subsection{Maximum of the Al-Salam--Chihara polynomials }
The Chebyshev polynomials $U_n$ and $ T_n$  on the interval $[-1, 1]$ are  bounded in absolute value by their value at $x=1$. Similar result holds for the $q$-Hermite polynomials, see
 \cite[Theorem 13.1.2]{ismail09classical}.
 For our proofs we need to extend this property to the Al-Salam--Chihara polynomials.
\begin{proposition} \label{L4.5+}
Let $q\in [0,1)$,
and $a,b$ be  real or a complex conjugate pair such that $0\leq ab\leq 1$ and $a+b\leq 0$. Then for all $x\in [-1,1]$ we have
\begin{equation}\label{Al_Salam_Chihara_bound*}
| Q_n(x;a,b \vert q)| \le Q_n(1;a,b \vert q).
\end{equation}
\end{proposition}

\begin{remark}\label{R:Ismail}
   A similar bound seems to be implied by \cite[(15.1.4)]{ismail09classical} without the explicit assumption that $a+b\leq 0$.
We %
note that \eqref{Al_Salam_Chihara_bound*} does not hold for arbitrary $a,b$: polynomial
$Q_2(x;q,q|q) =3 q^3+q^2+q-1-4 (q+1) q x+4 x^2$ has maximum at $x=-1$, which exceeds the value at $x=1$ by $8q(1+q)$.
\end{remark}
\begin{corollary}
Let $q\in [0,1)$,
$r\in[0,1]$, $\alpha\in [-\pi/2,\pi/2]$. Denote $a=-re^{\i \alpha}$ and $b=-re^{-\i \alpha}$. Then for all $x\in [-1,1]$ bound \eqref{Al_Salam_Chihara_bound*} holds.
\end{corollary}

\begin{proof}
  Indeed, $a+b=-r \cos \alpha\leq 0$ and $ab=r^2\in[0,1]$.
\end{proof}
\begin{proof}[Proof of Proposition \ref{L4.5+}] We will write $Q_n$ as a linear combination
\begin{equation*}\label{bound_proof1*}
 Q_n(x;a,b|q) =
\sum\limits_{m=0}^n a(n,m) T_m(x),
\end{equation*}
of the Chebyshev polynomials  $T_m(\cos \theta)=\cos (m \theta)$   of the first kind, $m=0,1,\dots$.
We are going to show that $a(n,m)\geq 0$ for all  $0\leq m\leq n$. Since for
$x=\cos\theta \in [-1,1]$ we have $|T_m(x)|\leq 1= T_m(1)$, this will imply the bound \eqref{Al_Salam_Chihara_bound*}.

To use \cite[Theorem 1]{szwarc1992connection},
we rewrite the recurrence for the Al-Salam--Chihara polynomials as
$$x Q_n(x)=\gamma_n Q_{n+1}(x)+\beta_n Q_n(x)+\alpha_n Q_{n-1}(x),$$
where
$$ \gamma_n=1/2,\quad \beta_n=\frac{a+b}{2}q^n,\quad \alpha_n=\frac{(1-q^n)(1-ab q^{n-1})}{2}.$$
The Chebyshev polynomials satisfy the recurrence
$$x T_n(x)=\gamma_n' T_{n+1}(x)+\beta_n' T_n(x)+\alpha_n' T_{n-1}(x),$$
where
$$  \gamma_n'=\frac{1}{2}(1+\delta_{n=0}) ,\quad \beta_n'=0,\quad \alpha_n'=\frac{1}{2}\delta_{n>0}.$$
It is then clear that the assumptions of \cite[Theorem 1]{szwarc1992connection} are satisfied.
Thus $a(m,n)\geq 0$ for all $0\leq m\leq n$, $n=0,1,\dots$, ending the proof.
\end{proof}

\subsection{Pointwise asymptotics near the end of the interval of orthogonality}
We will need pointwise asymptotics for the Al-Salam--Chihara polynomials at the right endpoint of the orthogonality interval. Such pointwise limits  have been studied for  orthogonal polynomials  both on finite intervals and on unbounded  domains \cite{aptekarev1993asymptotics,baik2003uniform,deift2001riemann,deift1999strong,ismail2013plancherel,ismail1986asymptotics,
ismail2005asymptotics,Ismail-Wilson-82,ismail2022orthogonal,kuijlaars2004riemann,lubinsky2020pointwise,nevai1984asymptotics}.
In particular,  \cite[Theorem 1]{aptekarev1993asymptotics} gives general conditions on the orthogonality measure and on the coefficients of the three term recurrence relation for the orthonormal polynomials $\{q_n(x)\}$, which imply
\begin{equation}
  \label{Apteka}
  \lim_{n\to \infty} \frac{1}{n^{\alpha+1/2}}\,q_n\left(1-\tfrac{u^2}{2n^2}\right)=\frac{J_\alpha(u)}{u^{\alpha}}
\end{equation}
uniformly with respect to $u$ on compact subsets of complex plane, where $J_\alpha$ is the Bessel function, $\alpha>-1$.

Lubinsky \cite[Theorem 1.1]{lubinsky2020pointwise}  proves a version of
\eqref{Apteka}, which can be re-written as
\begin{equation}
  \label{Lubinsky}
  \lim_{n\to \infty} \frac{q_n\left(1-\frac{u^2}{2n^2}\right)}{q_n(1)}=2^\alpha\Gamma(\alpha+1)\frac{J_\alpha(u)}{u^{\alpha}}
\end{equation}
uniformly on compact subsets of complex plane.  For our density function \eqref{Q-w}, the assumptions in \cite[Theorem 1.1]{lubinsky2020pointwise}   are satisfied with $\alpha=1/2$.
Indeed,
one can re-write \eqref{Q-w} as
\begin{equation}\label{Q-w+}
g(x) =  \frac{2\sqrt{1-x^2}(q,ab;q)_\infty\left|(q e^{2 \i \theta};q)_\infty\right|^2}{\pi\left|(a e^{i\theta},be^{\i \theta};q)_\infty\right|^2}, \; x=\cos \theta.
\end{equation}
  Note that for $\alpha=1/2$, the right-hand side of \eqref{Lubinsky} simplifies to  $u^{-1}\sin u$ and that the limit under normalization \eqref{Apteka} is discussed in
\cite[Theorem 1.3]{lubinsky2020pointwise}.

Pointwise asymptotics for the Al-Salam--Chihara polynomials  at the endpoints of the interval of orthogonality follows from the above results. However,
  our proof is quite direct, and the approach extends to the case of varying orthogonality measures, so we include it for completeness.
Our asymptotic result for the   Al-Salam-Chihara polynomials with varying orthogonality measure as  $q\nearrow 1$ seems to be new, in particular,   it is not covered by the results in \cite{levin2020local}.

The following asymptotics holds for fixed $q<1$.
\begin{theorem}\label{Thm:3.1}
Fix $a,b$ real or complex conjugate with $|a|,|b|<1$, $q\in[0,1)$  and $u>0$. Let $\{Q_n\}$ be the Al-Salam--Chihara polynomials \eqref{Al-Salam}. If  $u_M\to u$,  %
then
\begin{equation}
    \label{P-R-limQ+}
    \lim_{M\to \infty} \tfrac1{M}\,Q_{M}\pp{1-\tfrac{u_M^2}{2M^2};a,b\middle\vert q} = \frac{\sin u }{u} \;
 \frac{(a,b;q)_\infty}{(q;q)_\infty }.
\end{equation}
\end{theorem}
\begin{remark}\label{R:Thm3.1}
   We will also need the case $u=0$, with
\begin{equation}\label{1*}\lim_{n\to\infty}
\frac{1}{n}Q_n(1;a,b|q)
=\frac{(a,b;q)_\infty}{(q;q)_\infty}.
\end{equation}
 We also note the following extension of bound \eqref{Al_Salam_Chihara_bound*}: there is a constant $C=C(a,b,q)$ such that for all $n=0,1\dots$  and $x\in[-1,1] $ we have
\begin{equation}\label{1*b}
   |Q_n(x;a,b|q)|\leq C  (n+1).
\end{equation}
\end{remark}

\begin{proof}[Proof of Theorem \ref{Thm:3.1}] %
The use of $M$ in \eqref{P-R-limQ+}  is  solely   to avoid notation clash when we use it in the proof of Theorem \ref{Thm-loc-lim}, so for the proof, we replace $M$ by $n$.

The series \eqref{Q-gen} converges for all $|t|<1$, thus we can write
\begin{equation}\label{AA0}
\frac{Q_n(x;a,b|q)}{(q;q)_n}=\frac{1}{2\pi \i} \oint_{|t|=r}
t^{-n-1} \frac{(at,bt;q)_\infty}{(e^{\i\theta} t,e^{-\i\theta} t;q)_\infty} d t,
\end{equation}
where $r\in (0,1)$ and the integration is done in the counterclockwise  direction. We fix now $x=\cos \theta \in (-1,1)$,
so
that $\theta \in (0,\pi)$ and we  note that the integrand
\begin{equation*}\label{A0}
 F(t):=t^{-n-1} \frac{(at,bt;q)_\infty}{(e^{\i\theta} t,e^{-\i\theta}t;q)_\infty}=
t^{-n-1} \frac{(at,bt;q)_\infty}{(e^{\i\theta} tq,e^{-\i\theta} tq;q)_\infty} \times \frac{1}{(1-e^{\i \theta} t)(1-e^{-\i \theta}t)}
\end{equation*}
has two simple poles at $t=e^{\pm \i \theta}$ inside the disk $|t|<q^{-1}$. By the Cauchy Residue Theorem we can write
\begin{equation}\label{Q_n_residues}
\frac{Q_n(x;a,b|q)}{(q;q)_n}=-{\textnormal{Res}}(F;e^{\i \theta})-{\textnormal{Res}}(F;e^{-\i \theta})
+\frac{1}{2\pi \i} \oint_{|t|=R}
t^{-n-1} \frac{(at,bt;q)_\infty}{(e^{\i\theta} t,e^{-\i\theta} t;q)_\infty} d t,
\end{equation}
where $R$ is %
in the interval $(1,q^{-1})$.

Recalling that $\left|1-|z|\right|\le |1-z|\le 1+|z|$ and   $|t|=R>1$ we get
$$
\left|\frac{(at,bt;q)_\infty}{(e^{\i\theta} t,e^{-\i\theta} t;q)_\infty}\right|\le \frac{(-|a|R,-|b|R;q)_\infty}{(R;\,q)^2_\infty}.
$$
Therefore, the third %
term  in \eqref{Q_n_residues} is bounded by
$$\frac{(-|a|R,-|b|R;q)_\infty}{(R;\,q)^2_\infty\,R^n}\stackrel{n\to\infty}{\longrightarrow}0.$$

Since under our assumptions  $\theta=\theta_n$
  with
 $\cos \theta_n=1-u_n^2/(2n^2)$ %
is such that
$n\theta_n\to u$, we see that
$$\frac{1}{n}{\textnormal{Res}}(F;e^{\pm \i \theta_n})=\frac{e^{\mp\i (n+1)\theta_n }}{n(e^{\pm\i \theta_n}-e^{\mp\i \theta_n})}
    \frac{(ae^{\pm\i \theta_n},be^{\pm \i \theta_n};q)_\infty}{(e^{\pm 2\i\theta_n} q;q)_\infty (q;q)_{\infty}}\stackrel{n\to\infty}{\longrightarrow} \frac{e^{\mp iu}(a,b;q)_{\infty}}{\pm 2\i u\,(q;\,q)_\infty^2}.
$$
Consequently, returning to \eqref{Q_n_residues} we get
$$
\tfrac1{n}\,Q_{n}\pp{1-\frac{u_n^2}{2n^2};a,b\middle|q}\stackrel{n\to\infty}{\longrightarrow} \frac{(e^{iu}-e^{-iu})(a,b;q)_{\infty}}{2\i u(q;q)_\infty}.
$$
\end{proof}
\begin{proof}[Proof of Remark \ref{R:Thm3.1}]
The limit \eqref{1*}  follows from \eqref{Qn-expl} and an elementary lemma that if $\alpha_n\to\alpha$, $\beta_n\to\beta$ then $n^{-1}\sum_{k=1}^n \alpha_n\beta_{n-k}\to \alpha\beta$.
(To verify the latter, write
$\tfrac1n\sum_{k=1}^n \alpha_k \beta_{n-k}=\tfrac1n\sum_{k=1}^n (\alpha_k-\alpha) \beta_{n-k}+\frac{\alpha}{n}\sum_{k=1}^n \beta_{k}\to 0+\alpha\beta$.)
  To prove  \eqref{1*b} we use \eqref{Berg-Ismail}, which by triangle inequality gives a bound with explicit constant:
\begin{equation*}
 \sup_{x\in[-1,1]} | Q_n(x;a,b|q)|%
 \leq \frac{ (-|a|,-|b|;q)_n}{(q;q)_n} (n+1).
\end{equation*}
\end{proof}

The following asymptotics holds  for  $q\nearrow 1$.
  The statement is somewhat cumbersome as we will need to apply this result to $a,b$ that vary with $q$.
\begin{theorem}\label{Thm:3.2}
  If $\wt a, \wt b$ are real in $[0,\infty)$, or a complex conjugate pair with $\re(\wt a)\geq 0$, then
with
$m=\floor {M x}+\floor{M\log\left(M\sqrt{(1+\wt a)(1+\wt b)}\right)}$,
 $q_M=e^{-2/M}$ and $a_M=-q_M \wt a$, $b_M:=-q_M\wt b$ we have
\begin{equation}\label{Q2Kab}
  \lim_{M\to \infty} \frac{(q_M;q_M)_\infty ^2}{M(  a_M,  b_M;q_M)_\infty} \;\frac{ Q_m(\cos(u/M);a_M,b_M\vert  q_M)}{ (q_M;q_M)_m}
  = K_{\i |u|} (e^{-x}), \; u\in\RR.
\end{equation}
\end{theorem}

\begin{proof}%
 The proof relies on two technical estimates that we will prove in Section \ref{Sec:SF}.
In the proof, for simplicity of notation we suppress the dependence of $q$ on $M$.
Fix $u\in {\mathbb R}$ and let $a=-q\wt a$, $b=-q\wt b$, where  $q=e^{-2/M}$.

In view of \eqref{AA0} with $r=q$, we can write
\begin{equation}\label{eq:131}
\frac{Q_m(\cos(u/M);a,b\vert  q)}{(q;q)_m} =
\frac{1}{2\pi \i} \oint_{|t|=q}
\frac{(at,bt;q)_{\infty}}{(q^{\i u/2}t,q^{-\i u/2}t;q)_{\infty}} t^{-m-1} d t,
\end{equation}
where we are integrating counterclockwise along a circle of radius $q<1$.
 Next we change the variable of integration $t=q^{z}$, so that
  $dt = \ln(q) q^z d z=-(2/M) t d z$,
and obtain
\begin{equation}\label{eq:133}
\frac{Q_m(\cos(u/M);a,b\vert  q)}{(q;q)_m} =
\frac{1}{M\pi \i} \int_{1-M\pi \i/2 }^{1+M\pi \i/2 }
\frac{(a q^{z},b q^{z};q)_{\infty}}{(q^{\i u/2+z},q^{-\i u/2+z};q)_{\infty}} e^{2mz/M} d z.
\end{equation}
From \eqref{eq:133}, in view of \eqref{q-Gamma},   we get
\begin{multline}\label{eq:133+}
\frac{(q;q)_\infty ^2}{M(a,b;q)_\infty} \;\frac{ Q_m(\cos(u/M);a,b\vert  q)}{ (q;q)_m}
\\
= \frac{1}{\pi \i} \int_{1-M\pi \i/2 }^{1+M\pi \i/2 }
\tfrac{\Gamma_q(\i u/2+z)\Gamma_q(-\i u/2+z)}{M^{2}(1-q)^{2-2z}}\; \tfrac{(-\wt a q^{1+z},-\wt b q^{1+z};q)_\infty } {(-\wt a q,-\wt b q;q)_\infty }\;e^{2 m z/M} d z
= \frac{1}{\pi \i} \int_{1-\i \infty }^{1+\i \infty } f_M(z)dz,
\end{multline}
where %
\begin{multline}\label{f_M}
  f_M(z)   = {\mathbf 1}_{\{|\im (z)|\leq M \pi/2\}}
  \frac{\Gamma_q(\i u/2+z)\Gamma_q(-\i u/2+z)(-\wt a q^{1+z},-\wt b q^{1+z};q)_\infty}{M^{2-2z}(1-q)^{2-2z}(-\wt a q,-\wt b q;q)_\infty}\;e^{2 m z/M-2z\log M}.
\end{multline}
We now write $z=1+\i s$. In order to use the dominated convergence theorem, we
verify that for every  $x,u\in\RR$
(recall that $m$ depends on $x$) there are constants $A,B,M_*$ such that for all $M>M_*$ and all real $s$ we have
 \begin{equation}
   \label{|fM|}
   |f_M(1+\i s)|\leq A e^{-B |s|}.
 \end{equation}
To verify \eqref{|fM|} we write
$|f_M(1+\i s)| = f\topp 1 \cdot f\topp 2 \cdot f\topp 3$,
where
\begin{align*}
 f\topp 1 &=  {\mathbf 1}_{\{|s|\leq M \pi/2\}}|\Gamma_q(1+\i s+\i u/2)\Gamma_q(1+\i s-\i u/2)|, \\
f\topp 2&= \frac{|(-\wt a q^{2+\i s},-\wt b q^{2+\i s};q)_\infty|}{|(-\wt a q,-\wt b q;q)_\infty|} ,\\
f\topp 3&=\left\vert\frac{e^{2(1+\i s) m /M}}{M^{2}(1-q)^{2-2(1+\i s)}}\right\vert
=e^{2 (m/M-\log M)}.
\end{align*}
From \eqref{GqGq} we see that there exist constants $A,B,M_*>0$ such that $f\topp 1\leq A e^{-B |s|}$ for all real $s$. From \eqref{*dc-2} we see that $f^{(2)}\leq 1$.
 With $$\frac{m}{M}=\frac{\floor {M x}+\floor{M\log\left(M\sqrt{(1+\wt a)(1+\wt b)}\right)}}{M}\leq
 x +  \log\,\sqrt{(1+\wt a)(1+\wt b)}+\log\,M,$$ we see that
   $$ f\topp 3\leq e^{2\pp{x+\log \sqrt{(1+\wt a)(1+\wt b)}}}.
   $$
Thus \eqref{|fM|} holds.

We can now apply the dominated convergence theorem and pass to the limit inside the integral \eqref{eq:133+}.
We fix $z=1+\i s$ and compute $\lim_{M\to\infty}f_M(z)$ by computing the limit for %
the factors in \eqref{f_M}.

Clearly, ${\mathbf 1}_{\{|\im (z)|<M \pi/2\}} \to 1$ as $M\to\infty$ and by \eqref{qG2G} we get
$$\lim_{q\to 1^-}\Gamma_{q}(\i u/2+z)\Gamma_{q}(-\i u/2+z)=\Gamma(\i u/2+z)\Gamma(-\i u/2+z).$$
Since $\lim_{M\to\infty}\,M(1-q_M)=2$, from \eqref{eq6} we get
$$
\lim_{M\to\infty}\frac{(-\wt a q_M^{1+z},-\wt b q_M^{1+z};q_M)_\infty } {
M^{2-2z}(1-q_M)^{2-2z}(-\wt a q_M,-\wt b q_M;q_M)_\infty }
=\tfrac14 \frac{2^{2z}} {\left((1+\wt a)(1+\wt b)\right)^{z} }.
$$
Finally,  since
$$
  \frac{m}{M}%
  =  x+ \log M+\log\sqrt{(1+\wt a)(1+\wt b)}+O(\tfrac{1}M),
$$
we get
$$
e^{2 m z/M -2z \log M} \sim  e^{2z x+2 z \log M+2z \log\sqrt{(1+\wt a)(1+\wt b)}-2z \log M}
=\left((1+\wt a)(1+\wt b)\right)^z e^{2z x}.
$$
Putting all the factors together,
$$
\lim_{M\to\infty} f_M(z)= \frac{1}{4}\Gamma(\i u/2+z)\Gamma(-\i u/2+z)  2^{2z}
 e^{2z x}.
$$
Thus by the dominated convergence theorem,
$$
\lim_{M\to\infty} \frac{1}{\pi \i} \int_{1-\i \infty }^{1+\i \infty } f_M(z)d z
=
       \frac{1}{4\pi \i} \int_{1-\i \infty }^{1+\i\infty }
 \Gamma(\i u/2+z)\Gamma(-\i u/2+z)  2^{2z}
 e^{2z x}dz.
$$
This completes the proof by the Mellin-Barnes type formula for the Bessel K function:
\begin{equation*}
  \label{BessK}
  K_{\i u}(e^{-x})=\frac{1}{4\pi \i} \int_{c-\i \infty}^{c+\i\infty} \Gamma(s+\i  {u}/{2}) \Gamma(s-\i  {u}/{2})  2^{2s}e^{2 x s} ds,
\end{equation*}
where $c>0$, $u,x\in\RR$.
(This is \cite[10.32.13]{NIST:DLMF} [10.32.13]  \url{https://dlmf.nist.gov/10.32.E13} with $\nu=\i u$, $z=e^{-x}>0$, and shifted variable of integration $t=s+\i u/2$.)
\end{proof}

 \section{Motzkin paths with $q$-number weights}\label{sec:SCI}
 We now return to the setting from Section \ref{sec:Intro}, providing additional details and further results for the case of edge weights
  $a_n=[n+2]_q$, $b_n=2\sigma [n+1]_q$, $c_n=[n]_q$.

We note that with the above choice of edge weights, the probability measure \eqref{Pr0}  is the same as \eqref{Pr}  with the boundary weights
$\alpha_n=\rho_0^n[n+1]_q$ and $\beta_n=\rho_1^n$.
Indeed, we note that for an upward step, we have
$a_{\gamma_{k-1}}=[\gamma_{k-1}+2]_q=[\gamma_k+1]_q$, for a horizontal step we have $b_{\gamma_{k-1}}=2\sigma [\gamma_{k-1}+1]_q=2\sigma[\gamma_k+1]_q$ and for a downward step we have $c_{\gamma_{k-1}}=[\gamma_{k-1}]_q=[\gamma_k+1]_q$. So
formula \eqref{w} gives $w(\vv \gamma)=(2\sigma)^{H(\vv \gamma)}\prod_{k=1}^L [\gamma_k+1]_q$ and thus
\eqref{Pr} gives
\eqref{Pr0}.

 Recurrence \eqref{3-step} becomes
 \begin{equation}
     \label{3-step-B}
     x p_n(x)=[n+2]_q p_{n+1}(x)+2\sigma [n+1]_q p_n(x)+[n]_q p_{n-1}(x).
 \end{equation}
 Then \eqref{L2} gives
 \begin{equation}
    \label{wt-pi-n}
    \|p_n\|_2^2=\prod_{k=1}^{n} \frac{[k]_q}{[k+1]_q} = \frac{1}{[n+1]_q} ,
 \mbox{ hence $\wt p_n(x)=[n+1]_q p_n(x)$.}
 \end{equation}
\begin{lemma}\label{P:supp}
  Suppose $0\leq \sigma\leq 1$. Polynomials $\{p_n\}$ are orthogonal with respect to the density
  supported on the interval  $ [A,B]$ with
   \begin{equation}\label{pAB}A=-2\frac{1-\sigma}{1-q}, \quad  B=2\frac{1+\sigma}{1-q} .\end{equation}
Moreover,
  \begin{equation}\label{pi+}
    \pi_n= p_n(B)= \frac{1}{ [n+1]_q} \sum_{k=0}^n \frac{(a;q)_k  }{
    (q;q)_k } \;\frac{ (b;q)_{n-k}}{
      (q;q)_{n-k} },
\end{equation}
where
\begin{equation}\label{a,b}
    a=-{q}\pp{\sigma+\i\sqrt{1-\sigma^2}}, \; b=-{q}\pp{\sigma-\i\sqrt{1-\sigma^2}}.
\end{equation}
\end{lemma}
We note that $a,b$ are the solutions of the system of equations
\begin{equation}\label{ab=q2}
 a+b=-2\sigma q, \quad ab=q^2.
\end{equation}
\begin{proof}
To determine the orthogonality measure as well as the sequence $\pi_n=p_n(B)$,
 we establish a connection between polynomials $\{p_n\}$  and
monic Al-Salam-Chihara polynomials  $\{Q_n\}$,   defined by  the three term recurrence \eqref{Al-Salam}
with complex conjugate parameters \eqref{a,b}.

It is straightforward to verify that if $\{p_n\}$ satisfy recurrence \eqref{3-step-B} then polynomials
\begin{equation*} \label{Q2p} Q_n(x;a,b|q):=\pp{ 1-q}^{n} [n+1]_q! \; p_n\pp{ 2\;\frac{x+\sigma}{1-q}},\; n=0,1,\dots
\end{equation*}
satisfy recurrence \eqref{Al-Salam} with parameters \eqref{a,b}.
Conversely, with $y=2\frac{x+\sigma}{1-q} $,
we have
\begin{equation}
  \label{p2Q}
  p_n(y)=\frac{1}{[n+1]_q (q,q)_n} Q_n(x;a,b|q).
\end{equation}
Therefore, with $g$ defined by \eqref{Q-w}  (recall, that $g$ is supported on $[-1,1]$), polynomials $\{p_n\}$ are orthogonal with respect to the weight function (probability density function)
\begin{equation}\label{p-dens}
   \frac{1-q}{2} \;g\pp{\tfrac{(1-q)x}{2}-\sigma},
\end{equation}
which is supported on the interval $[A,B]$ given by \eqref{pAB}.  Combining  \eqref{Qn-expl} and \eqref{p2Q} we see that
\begin{equation}\label{pin}\pi_n=p_n(B)=\frac{Q_n(1;a,b|q)}{[n+1]_q(q;q)_n}\end{equation}
is given by \eqref{pi+}, see \eqref{Qn-expl}.
\end{proof}
With the above preparations, we can now prove Theorem \ref{Thm:1}.
\begin{proof}[Proof of Theorem \ref{Thm:1}]
We apply Theorem \ref{T-L}.  With $\alpha_n=\rho_0^n[n+1]_q$ and
$\beta_n=\rho_1^n$, conditions
 $(A_1)$ and $(A_2)$ from Assumption \ref{A-A1A2}   hold by  \eqref{p2Q}, \eqref{wt-pi-n}  and \eqref{1*b}.
The assumptions on the support of the orthogonality measure follow from Lemma \ref{P:supp}.
 The probabilities $\pi_n=\frac{s_n}{[n+1]_q}$ are given by \eqref{pi+}, so the transition matrix is just a recalculation of \eqref{trans-geX}.
Recalling that $\wt\pi_n=[n+1]_q\pi_n$, see \eqref{wt-pi-n}, the initial laws
\eqref{ini-0-laws} arise from \eqref{ini-law}.
Finally,   the formulas for the normalization constants are calculated from the relation \eqref{p2Q} by applying \eqref{Q-gen} with $t=\rho_j$, $j=1,2$, and $\theta=0$.
\end{proof}

 \subsection{Local limit theorems}
 We will  deduce Theorem \ref{Thm:2} and Theorem \ref{Thm:3} from
two  local limit theorems.  These theorems use \eqref{IR-2} to study convergence of the transition probabilities for the  Markov chain $\{X_k\}$
from %
Theorem \ref{Thm:1} under the 1:2 scaling of space and time.

We have the following local limit theorem for fixed $q$.
\begin{theorem}\label{Thm-loc-lim} Let  $\{X_k\}$ be the Markov chain
 introduced in Theorem \ref{Thm:1}.
If $0\leq q<1$, and $0<\sigma\leq 1$  then
for $x,y, t>0$ we have
   \begin{equation}
    \label{toprove1}
   \lim_{N\to\infty}\sqrt{N}\, \Pr\pp{X_{\floor{Nt}}=\floor{y\sqrt{N}}\middle|X_0=\floor{x\sqrt{N }}}
   = \frac{y}{x}\,\g_{\frac{t}{1+\sigma}}(x,y), %
   \end{equation}
   where  $\g_t$, $t>0$, is given in  \eqref{BMhit0}.

   Furthermore, if   %
   $\rho_0=1-\C/\sqrt{N}$  varies with $N$ for some fixed constant $\C>0$, then  denoting by $X_0\topp N$  a random variable with the %initial
   law \eqref{ini-0-laws}, for $x>0$ we have
   \begin{equation}\label{ini-gamma}
     \lim_{N\to\infty}\sqrt{N}\, \Pr\pp{X_{0}^{(N)}=\floor{x\sqrt{N}}}
   =\C^2 x e^{-\C x}.
   \end{equation}
\end{theorem}

\begin{proof}
We use \eqref{IR-2} with
\begin{equation}
  \label{kmn}
  k=\floor{Nt}, \quad m=\floor{x\sqrt{N}}, \quad n=\floor{y\sqrt{N}}.
\end{equation}
 (To avoid notation clash, we shall use $w$ and $\wt w$  instead of $x,y$ for the variables of integration, and in formulas such as \eqref{p2Q}.)

In view of \eqref{1*} and  \eqref{p2Q}, it is clear that for $x,y>0$ we have
\begin{equation}\label{y/x}\lim_{N\to\infty}\frac{\pi_{\floor{y\sqrt{N}}}}{\pi_{\floor{x\sqrt{N}}}}=\frac {y}{x}.\end{equation}
Indeed, the constant on the right-hand side of \eqref{1*}  is non-zero, as  our $a,b$ have modulus $q<1$, see \eqref{a,b}.
    So, see \eqref{p-dens} and \eqref{pAB}, we need to find the limit of the expression
\begin{multline*}
      \frac{1-q}{2}\frac{\sqrt{N}}{B^k}\int_A^B \wt w ^k p_m(\wt w)\wt p_n(\wt w)g\pp{\tfrac{1-q}{2}\wt w-\sigma}   d\wt w
 \\ = \frac{\sqrt{N}}{(1+\sigma)^k[m+1]_q} \int_{-1}^1
   (w+\sigma)^k \frac{Q_m(w)}{(q;q)_m}\frac{Q_n(w)}{(q;q)_n} g(w) d w.
\end{multline*}

(Here, we used  \eqref{p2Q} and identity/substitution $\wt w/B=(w+\sigma)/(1+\sigma)$.)

Next we observe that the integral  over the interval $[-1,0]$ is negligible.
Indeed,  on this interval
$|w+\sigma|\leq \sigma\vee(1-\sigma)=r(1+\sigma)
 $, where the latter defines $r$.
Together  with \eqref{1*b} this shows that
 $$
 \tfrac{\sqrt{N}}{[m+1]_q} \left|\int_{-1}^0
    \pp{\tfrac{w+\sigma}{1+\sigma}}^k
   \tfrac{Q_m(w)}{(q;q)_m}\tfrac{Q_n(w)}{(q;q)_n} g(w) dw  \right| %
   \leq C \sqrt{N} (m+1)(n+1) r^k \int_{-1}^0 g\pp{w}dw \to 0
 $$
 as  $N\to\infty$, since $k, m, n $ are given by \eqref{kmn}  and $r\in[0,1)$.

It remains to analyze the  case of  the integral over $[0,1]$.
Changing the variable of integration  by setting  $u=\sqrt{2N(1-w)}$, i.e. $w=1-\frac{u^2}{2N}$,  we get
\begin{multline}\label{2nd-lim2}
  \frac{\sqrt{N}}{[m+1]_q} \int_{0}^1
   \left(\frac{w+\sigma}{1+\sigma}\right)^k\, \frac{Q_m(w)}{(q;q)_m}\frac{Q_n(w)}{(q;q)_n} g(w) dw
\\
=\frac{1}{[m+1]_q} \int_{0}^{\sqrt{2N}}
   \pp{1-\tfrac{u^2}{2(1+\sigma)N}}^k \frac{Q_m\pp{1-\tfrac{u^2}{2N}}}{\sqrt{N}(q;q)_m} \frac{Q_n\pp{1-\tfrac{u^2}{2N}}}{\sqrt{N}(q;q)_n} \sqrt{N} g\pp{1-\tfrac{u^2}{2N}}   u du.
  \end{multline}

Clearly, $\frac{1}{[m+1]_q}\to 1-q$.
The next observation is that for $u\in[0,\sqrt{2N}]$ and
$k=\floor{Nt}$, %
we have
\begin{equation}\label{expexp}\pp{1-\frac{u^2}{2(1+\sigma)N}}^k %
\sim e^{- \frac{u^2 t}{2(1+\sigma)}} \mbox{ as $N\to\infty$}.\end{equation}

 Now we refer to \eqref{Q-w+} for $g\pp{1-\tfrac{u^2}{2N}}$. Noting  that $\theta\sim \frac{u}{\sqrt{N}}\sim \sqrt{1-(1-\frac{u^2}{2N})^2}$,  we get %
\begin{equation}\label{l*l}
  \sqrt{N}   g\pp{1-\frac{u^2}{2N}} \sim \frac{2\,u }{\pi  } (q;q)_\infty^3 \frac{(ab;q)_\infty}{(a,b;q)^2_\infty}
    = \frac{2 u }{\pi (1-q)} \;\frac{(q;q)_\infty^4  }{(a,b;q)^2_{\infty}}
\end{equation}
for $u\in[0,\sqrt{2N}]$,
where the last equality is by \eqref{ab=q2}.

We now determine the asymptotics for the remaining factors of the integrand on the right-hand side of \eqref{2nd-lim2}. We
apply Theorem \ref{Thm:3.1} with $M=\floor{\sqrt{N}x}\to \infty$ and with $u_M/M:=u/\sqrt{N}$ so that
$u_M=u M /\sqrt{N}= u x +O(1/M)$. Recalling \eqref{kmn}, we see that $m=M$ and thus in view of \eqref{P-R-limQ+} we have
\begin{equation}\label{Qm2sin}
    \frac{Q_m\pp{1-\frac{u^2}{2N}}}{\sqrt{N}} =
\frac{\floor{\sqrt{N}x}}{\sqrt{N}} \;
\frac{
Q_{M}\pp{1-\frac{u_M^2}{2M^{2}}}}{M}
\to x\, \frac{\sin (u x)}{u x} \;
 \frac{(a,b;q)_\infty}{(q;q)_\infty } = \frac{\sin (u x)}{u } \;
 \frac{(a,b;q)_\infty}{(q;q)_\infty }
\end{equation}
 as $N\to\infty$. The same argument with $n$ from \eqref{kmn} gives
\begin{equation}\label{Qn2sin}
\frac{Q_n\pp{1-\frac{u^2}{2N}}}{\sqrt{N}} \to \frac{\sin (u y)}{u } \;
 \frac{(a,b;q)_\infty}{(q;q)_\infty }  \mbox{ as $N\to\infty$}.
\end{equation}
To verify uniform integrability, we apply elementary bounds
$$(q;q)_\infty\leq |(z;q)_\infty|\leq (-q;q)_\infty, \quad |z|\leq q $$
within \eqref{Q-w+}. We get
$$\sqrt{N}\,g\pp{1-\tfrac{u^2}{2N}}\leq  \sqrt{N}\sqrt{1-\left(1-\tfrac{u^2}{2N}\right)^2} \; \frac{2(-q;q)_\infty^2}{\pi(1-q)(q;q)_\infty}\;\leq C(q) |u|.$$
 Together with  an elementary  bound $(1-x/N)^k\leq e^{-x t}$ for $0\leq x\leq N$  and the fact that $[m+1]_q\to 1/(1-q)$ this shows that
the integrand on the right-hand side of \eqref{2nd-lim2} is bounded by an integrable function $C u^2 e^{-c u^2}$, so we can use dominated convergence theorem to pass to the limit under the integral.
 Combining \eqref{y/x}, \eqref{2nd-lim2}, \eqref{expexp},  \eqref{l*l}, \eqref{Qm2sin} and \eqref{Qn2sin}  %
we conclude that,  the left-hand side of  \eqref{toprove1} becomes
\begin{multline*}%
  \frac{y}{x}\,\frac{(q;q)_\infty^4}{(a,b;q)_\infty^2} %
   \frac{2 }{\pi}  \int_{0}^\infty  \lim_{N\to\infty} \left(\ind_{u<\sqrt{2 N}}
   \pp{1-\frac{u^2}{2(1+\sigma)N}}^{\floor{Nt}}
   \frac{Q_m\pp{1-\frac{u^2}{2N}}}{\sqrt{N}(q;q)_m}
   \frac{Q_n\pp{1-\frac{u^2}{2N}}}{\sqrt{N}(q;q)_n}\right)\,  u^2 \,du
   \\
   =\,\frac{y}{x}\, \frac{2}{\pi}\int_0^\infty e^{-\frac{u^2t}{2(1+\sigma)}}\sin(ux)\sin(uy)\, du =   \frac{y}{x}\,\sqrt{\frac{1+\sigma}{2\pi t}}\,\left(e^{-\frac{(y-x)^2(1+\sigma)}{2t}}-e^{-\frac{(y+x)^2(1+\sigma)}{2t}}\right)
\end{multline*}
as required.

To prove \eqref{ini-gamma} we use \eqref{1*}.
Using notation \eqref{kmn}, from \eqref{ini-0-laws} and  \eqref{Qn-expl} %
  we get
\begin{align}
  &\sqrt{N}\Pr(X_0\topp{N}=m)=\sqrt{N} \frac{(\rho_0;q)_\infty^2}{(a\rho_0,b\rho_0;q)_\infty}
  \rho_0^m  \frac{ Q_m(1;a,b|q)}{ (q;q)_m}\label{NPr}\\
  &=
  N  \left(\frac{\C}{\sqrt{N}}\right)^2\,\frac{\floor{x\sqrt{N}}}{\sqrt{N}}\,\left(1-\frac{\C}{\sqrt{N}}\right)^{\floor{x\sqrt{N}}}\,  \frac{(q\rho_0;q)_\infty^2}{(q;q)_m}
 \;  \frac{ Q_m(1;a,b|q)}{ m \,(a\rho_0,b\rho_0;q)_\infty} \to \C^2x\,e^{-\C x}
\nonumber
\end{align}
as $N\to\infty$.
\end{proof}

Next we prove a local limit theorem with $q\nearrow 1$.
 Note that if parameter  $q$ varies with $N$, then the law of the Markov chain $\{X_k\}$
 from %
 Theorem \ref{Thm:1} also varies with $N$, so
  we  denote this process by $\{X_k\topp N\}$.

For $z\in\RR$ we denote %
 \begin{equation}\label{zN}
 \Jz{z}{N}=\floor{z\sqrt{N}}+\floor{\sqrt{N}\,\log\,\sqrt{N\,2\log(1+\sigma)}}.
 \end{equation}
 With this notation we
 have the following version of Theorem \ref{Thm-loc-lim}  for $q\nearrow 1$:
\begin{theorem}\label{Thm-q=1}
Let  $\{X_k\topp N\}$ be the Markov chain
 introduced in Theorem \ref{Thm:1}.
If $0<\sigma\leq 1$ is fixed %
and $q=e^{-2/\sqrt{N}}$  then
for $x,y\in\RR$, $t>0$ we have
    \begin{equation}\label{toprove*}
   \lim_{N\to\infty}\sqrt{N} \,\Pr\Big(X_{\floor{ Nt}}^{(N)}=\Jz{y}{N}
   \Big\vert\; X_0^{(N)} =\Jz{x}{N}
   \Big)
    = \frac{K_0(e^{-y})}{K_0(e^{-x})}\; \p_{\frac{t}{1+\sigma}}(x,y),
   \end{equation}%
   where $\p_t$, %
   $t> 0$,
   is given in \eqref{Yak}.

In addition, if $\rho_0=e^{-\C/\sqrt{N}}=q^{\C/2}$  for some fixed constant $\C>0$, then
\begin{equation}\label{ini-K}
  \lim_{N\to\infty}\sqrt{N} \Pr\pp{X_{0}^{(N)}=\Jz{x}{N}}=%
  \frac{4}{2^\C \Gamma(\C/2)^2} K_0(e^{-x}).
\end{equation}
\end{theorem}

\begin{proof}  We write the parameters \eqref{a,b} in trigonometric form $a=-q e^{\i\alpha}$, $b=-q e^{-\i\alpha}$ with $\sigma=\cos \alpha$, $\alpha\in(-\pi/2,\pi/2)$.

First, we prove \eqref{ini-K}.
As in \eqref{NPr} we have
$$
 \sqrt{N}\,\Pr(X_0^{(N)} =\Jz{x}{N})=\sqrt{N} \frac{(\rho_0;q)_\infty^2}{(a\rho_0,b\rho_0;q)_\infty}
  \rho_0^{\Jz{x}{N}}%
  \frac{ Q_{\Jz{x}{N}}(1;a,b|q)}{ (q;q)_{\Jz{x}{N}}}.
$$
Denote $M=\sqrt{N}$. %
Then with $\widetilde a=e^{\i\alpha}$ and $\widetilde b=e^{-\i\alpha}$ we see that $m$ as defined in Theorem \ref{Thm:3.2} assumes the form $m=\Jz{x}{N}$. Consequently, we can rewrite the above expression as follows
$$
 \sqrt{N}\,\Pr(X_0^{(N)} =\Jz{x}{N})= \left(Me^{-\frac{m}M}\right)^{\C}\, M^{2-\C}\, \left(\frac{(\rho_0;q)_\infty}{(q;q)_\infty}\right)^2 \frac{(a,b;q)_\infty}{(a\rho_0,b\rho_0;q)_\infty}  \times J_M,$$
 where
 \begin{equation}\label{JM}
  J_M=\frac{(q;q)_\infty^2}{M (-q\widetilde a,-q\widetilde b;q)_\infty}\frac{ Q_{m}(1;-q\widetilde a,-q\widetilde b|q)}{ (q;q)_{m}}\to K_0(e^{-x})\quad\mbox{as }\;N\to\infty,
  \end{equation}
 with the limit deduced from  \eqref{Q2Kab}. Clearly, %
 \begin{equation}\label{MeM}
  Me^{-\frac{m}M}=\sqrt{N} e^{-\frac{\Jz{x}{N}}{\sqrt{N}}}\to \frac{e^{-x}}{\sqrt{2(1+\sigma)}}\quad \mbox{as }\;N\to\infty.
  \end{equation}

  Moreover, recalling \eqref{q-Gamma}, we can write
  \begin{equation}\label{M2c}
  M^{2-\C}\, \left(\frac{(\rho_0;q)_\infty}{(q;q)_\infty}\right)^2=M^{2-\C}\, \left(\frac{(q^{\frac{\C}{2}};q)_\infty}{(q;q)_\infty}\right)^2=\frac{(M(1-q))^{2-\C}}{\Gamma_q\left(\C/2\right)^2}\to\frac{2^{2-\C}}{\Gamma\left(\C/2\right)^2}
  \end{equation}
  as $N\to\infty$, where the last limit follows from \eqref{qG2G}.

  Finally, relying on \eqref{eq6}, we get %
  \begin{multline}\label{Ramaj}
  \frac{(a,b;q)_\infty}{(a\rho_0,b\rho_0;q)_\infty} =\frac{(-\wt a q,-\wt b q;q)_\infty}{(-\wt a q^{1+\C/2},-\wt b q^{1+\C/2};q)_\infty} \\=\frac{(1+\widetilde a q^{\frac{\C}{2}})(1+\widetilde b q^{\frac{\C}{2}})}{(1+\widetilde a)(1+\widetilde b)}\,\frac{(-\widetilde a;q)_{\infty}}{(-\widetilde a q^{\frac{\C}{2}};q)_{\infty}}\,\frac{(-\widetilde b;q)_{\infty}}{(-\widetilde b q^{\frac{\C}{2}};q)_{\infty}}
   \to (1+e^{\i \alpha})^{\C/2}(1+e^{-\i \alpha})^{\C/2}
=  2^{\C/2}(1+\sigma)^{\C/2}
  \end{multline}
as  $(1+e^{\i \alpha})^{\C/2}(1+e^{-\i \alpha})^{\C/2}=(2+2\cos\alpha)^{\C/2}$. %
Combining \eqref{JM}, \eqref{MeM},  \eqref{M2c} and \eqref{Ramaj} we get \eqref{ini-K}.

 Second, we prove \eqref{toprove*}.
 Our starting point is formula \eqref{IR-2}, which we use with
 \begin{equation}
   \label{kmn-log}
   m:=\Jz{x}{N},\quad
   n:=\Jz{y}{N},\quad
   k=\floor{N t}.
 \end{equation}
 Even though $x,y\in\RR$, note that $m,n$ are non-negative integers for large enough $N$.
 In particular, we note that $m/n\to 1$ and $m/k\to 0$ as $N\to\infty$. It is also clear that
 $[n+1]_q\sim [m+1]_q\sim \sqrt{N}/2$.
  (To avoid notation clash, we  use $w$ and $\wt w$  instead of $x,y$ for the variables of integration, or in formulas like  \eqref{p2Q}.)

  From   \eqref{IR-2}, we get
\begin{multline}\label{6*g}
 \sqrt{N} \Pr(X_k=\Jz{y}{N}|X_0=\Jz{x}{N}) =\frac{\pi_n}{\pi_m} \frac{\sqrt{N}}{B^k}\int_{A}^B \wt w^k p_m(\wt w) [n+1]_q p_n(\wt w) g(\wt w) d\wt w
  \\= \tfrac{Q_n(1;a,b|q)}{ (q;q)_{n}} \; \tfrac{(q;q)_{m}}{Q_m(1;a,b|q)}\; \tfrac{\sqrt{N}} %
  {[n+1]_q}
  \int_{-1}^{1} \frac{\pp{w+\sigma}^k}{(1+\sigma)^k} \frac{Q_m\pp{w;a,b|q}}{(q;q)_{m}}\, \frac{Q_n\pp{w;a,b|q}}{(q;q)_n}\,g(w)\,  d w,
  \end{multline}
  where the last equality follows from \eqref{pin} and \eqref{p2Q}.

By Theorem \ref{Thm:3.2} used with $M=\sqrt{N}$, $u=0$,  and $m=\Jz{y}{N}$ or $m=\Jz{x}{N}$ we get
\begin{multline}\label{K/K}
\lim_{N\to\infty}  \frac{Q_n(1;a,b|q)}{(q;q)_{n}} \; \frac{(q;q)_{m}}{Q_m(1;a,b|q)}
\\=
\lim_{N\to\infty} \frac{(q;q)_\infty^2 Q_{\Jz{y}{N}}\,(1;a,b|q)}{\sqrt{N}  (a,b;q)_\infty\,(q;q)_{\Jz{y}{N}}} \times  \lim_{N\to\infty} \frac{\sqrt{N}  (a,b;q)_\infty(q;q)_{\Jz{x}{N}}}{(q;q)_\infty^2Q_{\Jz{x}{N}}(1;a,b|q)}
=\frac{K_0(e^{-y})}{K_0(e^{-x})}.
\end{multline}
This gives the pair of Bessel functions $K_0$ on the right-hand side of \eqref{toprove*}. %

Next, we show that the contribution from the integral in \eqref{6*g} over $[-1,1/2]$ is negligible.
Note that for $w$ from this interval $|w+\sigma|<(1/2+\sigma)\vee(1-\sigma)=r  (1+\sigma)$, where the latter defines $r$.
 Similarly   as in the proof of Theorem \ref{Thm-loc-lim},  the contribution of the integral over $[-1,1/2]$ is bounded by %
\begin{multline*}
  C  r^k\, \frac{\sqrt{N}}{[n+1]_q}
  \int_{-1}^{1/2}
   \frac{|Q_m\pp{w;a,b|q}|}{(q;q)_{m}} \frac{|Q_n\pp{w;a,b|q}|}{(q;q)_n} g\pp{w}  d w\leq
   \\  C  r^k\, \frac{\sqrt{N}}{2 [n+1]_q}
    \int_{-1}^1
   \pp{\frac{Q_m^2\pp{w;a,b|q}}{(q;q)_{m}^2}+ \frac{Q_n^2\pp{w;a,b|q}}{(q;q)_n^2}} g\pp{w}  d w
   \\= C  r^k\, \frac{\sqrt{N}}{2 [n+1]_q}\pp{[m+1]_q+[n+1]_q}\to 0.
\end{multline*}
Here $C$ is a constant from boundedness of the  convergent sequence in \eqref{K/K}, and the last equality follows from \eqref{wt-pi-n} and \eqref{p2Q}. To compute the limit
 we used the observation that $[n+1]_q\sim [m+1]_q\sim \sqrt{N}/2$, $k\sim N t$ and $r\in[0,1)$.

Substituting $w=\cos(u/\sqrt{N})$ into \eqref{6*g} and discarding the non-contributing part of the integral, we have
\begin{equation}\label{6*h}
 \sqrt{N} \Pr(X_k=\Jz{y}{N}|X_0=\Jz{x}{N}) \sim \frac{K_0(e^{-y})}{K_0(e^{-x})}   \int_0^\infty\,\mathbf 1_{(0,\,\frac{\pi}{3}\sqrt{N})}(u) f_N(u)\,du,
 \end{equation}
 where  %
 $$
  f_N(u)= %
 \tfrac{1}{[n+1]_q} \pp{\tfrac{\sigma+\cos\frac{u}{\sqrt{N}}}{1+\sigma}}^k
   \tfrac{Q_m\pp{\cos\frac{u}{\sqrt{N}};a,b|q}}{(q;q)_{m}}\, \tfrac{Q_n\pp{\cos\frac{u}{\sqrt{N}};a,b|q}}{(q;q)_n} g\pp{\cos(\tfrac{u}{\sqrt{N}})} \sin(\tfrac{u}{\sqrt{N}}).
  $$
  Referring first to \eqref{Q-w} and then using $(1-q)(ab;q)_\infty=(q;q)_\infty$ and \eqref{q-Gamma} we get %
$$g\pp{\cos(\tfrac{u}{\sqrt{N}})} \sin(\tfrac{u}{\sqrt{N}})=
\frac{(q,a b;q)_\infty|(q^{\i u};q)_\infty|^2}{2\pi |(a q^{\i u/2},b q^{\i u/2};q)_\infty|^2} %
=\frac{(q;q)_\infty^4(1-q)}{2\pi |\Gamma_q(\i u)|^2\,|(a q^{\i u/2},b q^{\i u/2};q)_\infty|^2}.$$
Thus $f_N$ can be written as  %
\begin{multline}
f_N(u)=%
\frac{N(1-q)}{[n+1]_q}\,
\left(\frac{\sigma+\cos \frac{u}{\sqrt{N}}}{1+\sigma}\right)^{\floor{Nt}}\,\frac{1}{2\pi\,|\Gamma_q(\i u)|^2} \\
\times   \frac {(q;q)_\infty^2 Q_{\Jz{x}{N}}(\cos \frac{u}{\sqrt{N}};a,b|q)}{\sqrt{N} (a,b;q)_{\infty}\,(q;q)_{\Jz{x}{N}}} \,\frac {(q;q)_\infty^2 Q_{\Jz{y}{N}}(\cos \frac{u}{\sqrt{N}};a,b|q)}{ \sqrt{N}(a,b;q)_{\infty}\,(q;q)_{\Jz{y}{N}}}\, \frac{(a,b;q)_{\infty}^2}{|(aq^{\i u/2},b q^{\i u/2};q)_{\infty}|^2}.
\end{multline}
Using $[n+1]_q\sim \sqrt{N}/2$ again, we see that the first factor is asymptotically constant,  ${N(1-q)}/{[n+1]_q}\sim 4$.
Furthermore,
$$
\left(\frac{\sigma+\cos \frac{u}{\sqrt{N}}}{1+\sigma}\right)^{\floor{Nt}}\to e^{-\frac{u^2 t}{2(1+\sigma)}}.
$$
{
%By Lemma \ref{Lem:Ram},
By Lemma \ref{L:Ram2}, $\Gamma_q(\i u)\to \Gamma(\i u)$. From Theorem \ref{Thm:3.2}, we see that for real $z$, recall \eqref{zN},
\begin{equation}\label{*7*} \frac {(q;q)_\infty^2}{\sqrt{N}(a,b;q)_{\infty}}\,\frac{Q_{\Jz{z}{N}}(\cos\frac{u}{\sqrt{N}};a,b|q)}{(q;q)_{\Jz{z}{N}}} \to  K_{\i u}(e^{-z}).
\end{equation}
 In view of Lemma \ref{Lem:Ram}  %
\begin{multline*}
\frac{(a,b;q)_{\infty}}{(aq^{
\i u/2},b q^{\i u/2};q)_{\infty}}= \frac{(-e^{\i\alpha}q;q)_{\infty}} {(-e^{\i\alpha}q^{1+\i u/2};q)_{\infty}}
\frac{(-e^{-\i\alpha}q;q)_{\infty}} {(-e^{-\i\alpha}q^{1+\i u/2};q)_{\infty}} %
\to
  (1+e^{\i \alpha})^{\i u/2} (1+e^{-\i \alpha})^{\i u/2} \\= (2+2\sigma)^{\i \frac{u}{2}}.
\end{multline*}
 Note that $|(2+2\sigma)^{\i \frac{u}{2}}|=1$, so this factor does not contribute to the limit of $f_N$.

}
To summarize, we see that %
\begin{equation}\label{*6*}
\lim_{N\to\infty}\,\mathbf 1_{(0,\,\frac{\pi}{3}\sqrt{N})}(u)\,f_N(u)=\mathbf 1_{(0,\infty)}(u)\;\frac{2}{\pi}\,e^{-\frac{u^2 t}{2(1+\sigma)}}\, K_{\i u}(e^{-x})\, K_{\i u}(e^{-y})\,\frac{1}{|\Gamma(\i u)|^2}.
\end{equation}

We now justify that one can pass to the limit  under the integral \eqref{6*h}.
For $0<u< \frac{\pi}{3}\,\sqrt{N}$ we have
$$\left(\frac{\cos(u/\sqrt{N})+\sigma}{1+\sigma}\right)^k\leq \left(1-\frac{u^2}{2(1+\sigma)N}\right)^{N t}
\leq C e^{-\frac{u^2t}{2(1+\sigma)}} $$
for some $C<2.5$.   Recall that $t,\sigma$ are fixed.

By  Proposition \ref{L4.5+}  and Theorem \ref{Thm:3.2}   we know that  for
real $z$, recall \eqref{*7*}, we have
$$
\left| \frac {(q;q)_\infty^2}{\sqrt{N}(a,b;q)_{\infty}}\,\frac{Q_{\Jz{z}{N}}(\cos \frac{u}{M};a,b|q)}{(q;q)_{\Jz{z}{N}}}\right|  \leq
\frac {(q;q)_\infty^2}{\sqrt{N}(a,b;q)_{\infty}}\,\frac{Q_{\Jz{z}{N}}(1;a,b|q)}{(q;q)_{\Jz{z}{N}}}\to  K_0(e^{-z}),
$$
 therefore the expression on the left-hand side above is uniformly bounded in $u,N$. By Lemma \ref{Prop-Alexey},
 there exist $A>0$, $B>0$ and $M^*$ such that
 \begin{equation*}\label{ab/qaqb'}
  {\mathbf 1}_{\{|u|<\pi \sqrt{N}/2 \}} \frac{(a,b;q)_{\infty}^2}{|(aq^{\i u/2},bq^{\i u/2};q)_{\infty}|^2}\leq A e^{B|u|}
 \end{equation*}
 for all $N\geq M_*^2$ and all $u\in\RR$. By Lemma \ref{P:gamma},
 $1/|\Gamma_q(\i u)|^2\leq A u^2 e^{B|u|}$ for $|u|<\sqrt{N}\pi /2$.
 Moreover, recalling that $n=\Jz{y}{N}$, the convergent sequence $N(1-q)/[n+1]_q$ is bounded.
 Combining all these bounds, we see that there are constants $A,B,M_*>0$ such that   for   all $N>M_*^2$  function $f_N(u)\mathbf 1_{(0,\frac{\pi}{3}\sqrt{N})}(u)$    is bounded  on $[0,\infty)$   by  the integrable function %
 $$A u^2 e^{B |u|}  e^{-\frac{u^2t}{2(1+\sigma)}},$$
so we can invoke the dominated convergence theorem.  Taking the limit \eqref{*6*} inside the integral in \eqref{6*h}, we see that
\begin{multline*}
 \lim_{N\to\infty} \sqrt{N} \Pr(X_k=\Jz{y}{N}|X_0=\Jz{x}{N}) %
  =
  \frac{K_0(e^{-y})}{K_0(e^{-x})}   \,\frac{2}{\pi}\int_0^\infty\,e^{-\frac{u^2 t}{2(1+\sigma)}}\, K_{\i u}(e^{-x})\, K_{\i u}(e^{-y})\,\frac{d u}{|\Gamma(\i u)|^2},
\end{multline*}
 which ends the proof by \eqref{Yak}.
\end{proof}
\subsection{Proofs of Theorems \ref{Thm:2} and \ref{Thm:3}}
Both proofs are very similar.
\begin{proof}[Proof of Theorem \ref{Thm:2}]
Denote by $\{X_k\topp N\}$ the Markov chain from the conclusion of Theorem \ref{Thm:1} with $\rho_0=e^{-\C/\sqrt{N}}$.
  For fixed $t_0=0<t_1<\dots<t_d$ and $x_0,\dots,x_d> 0$ by
  \eqref{ini-gamma}, \eqref{toprove1}  of Theorem \ref{Thm-loc-lim}
   and the Markov property, we have
    \begin{multline*}
      \lim_{N\to\infty}  N^{(d+1)/2} \,\Pr\pp{\frac{X_{\floor{t_iN}}\topp N}{\sqrt{N}}=\tfrac{\floor{x_i\sqrt{N}}}{\sqrt{N}},\,i=0,1,\ldots,d}
=  \C^2e^{-\C x_0}\,x_d\,\prod_{j=1}^d\,\g_{\frac{t_j-t_{j-1}}{1+\sigma}}(x_{j-1},x_j),
    \end{multline*}
 where  we recall that $\g_t$ is defined in \eqref{BMhit0}.   %
    This proves local  convergence to  the  joint density
  of the vector $(\xi_0,\xi_{t_1},\dots,\xi_{t_d})$,   see \eqref{xi-tr} and \eqref{xi0}.

By \cite[Theorem 3.3]{billingsley99convergence},
the convergence of finite-dimensional distributions as stated in \eqref{12conv} follows.
\end{proof}
\begin{proof}[Proof of Theorem \ref{Thm:3}]
Denote by $\{X_k\topp N\}$ the Markov chain from the conclusion of Theorem \ref{Thm:1} with $\rho_0=e^{-\C/\sqrt{N}}$ and $q=e^{-2/\sqrt{N}}$.

Clearly, the centering sequence on the left-hand side of \eqref{13conc} in Theorem \ref{Thm:3} can be replaced by $\floor{\sqrt{N}\log\sqrt{2N(1+\sigma)}}$, $N\ge 1$.

  For fixed $t_0=0<t_1<\dots<t_d$ and $x_0,\dots,x_d\in\RR$ by
  \eqref{ini-K}, \eqref{toprove*} of Theorem \ref{Thm-q=1}    and the Markov property we have
       \begin{multline*}
      \lim_{N\to\infty} N^{\frac{d+1}2} \Pr\pp{\frac{X_{\floor{t_iN}}\topp N-\floor{\sqrt{N}\log \sqrt{2 N(1+\sigma)}}}{\sqrt{N}}=\frac{\floor{x_i\sqrt{N}}}{\sqrt{N}},\,i=0,1,\ldots,d}
      \\\stackrel{\eqref{zN}}{=}
    \lim_{N\to\infty} N^{\frac{d+1}2} \Pr\pp{X_{\floor{t_iN}}\topp N=\Jz{x_i}{N},\,i=0,1,\ldots,d}
= \frac{4}{2^\C \Gamma(\C/2)^2} K_0(e^{-x})\prod_{j=1}^d \p_{\frac{t_j-t_{j-1}}{1+\sigma}}(x_{j-1},x_j),
    \end{multline*}
 where  we recall that $\p_t$ is defined in \eqref{Yak}.   This proves local  convergence to  the  joint density  of the vector $(\zeta_0,\zeta_{t_1},\dots,\zeta_{t_d})$, see \eqref{zeta-tr} and \eqref{zeta0}.

By \cite[Theorem 3.3]{billingsley99convergence} the convergence of finite-dimensional distributions as stated in  \eqref{13conc} follows.
\end{proof}

\section{Auxiliary results on special functions}\label{Sec:SF}

\subsection{Some useful limits for $q\nearrow 1$}\label{sec:Ramanujan}

 The following result appears in Ch. 16 of Ramanujan's notebook, see \cite[Entry 1]{adiga1985ramanujan}, \cite[ p.  13, Entry 1(i)]{berndt-ramanujan} or \cite[(I34)]{gasper2004basic} for the statement. The proof in \cite[Proposition A.2]{koornwinder1990jacobi} is for real $\la$, which is not enough for our purposes. The proof in \cite[ p.  13]{berndt-ramanujan} uses analytic continuation from $|z|<1$.  We note that we always use the principal branch of the logarithm and of the power function.
\begin{lemma}[Ramanujan]\label{Lem:Ram} For all $\lambda \in {\mathbb C}$ and $z$ in cut complex plane ${\mathbb C}\setminus [1,\infty)$ we have
\begin{equation}\label{eq6}
\frac{(z;q)_{\infty}}{(z q^{\lambda}; q)_{\infty}} \to (1-z)^{\lambda}
\end{equation}
as $q\to 1^-$. The convergence is uniform in $z$ on compact subsets of
${\mathbb C}\setminus [1,\infty)$.
\end{lemma}
\begin{proof}
Denote the left-hand side of \eqref{eq6} by $f_q(z)$.
Then
\begin{multline*}
  h_q(z):=\frac{f_q'(z)}{f_q(z)}=\sum\limits_{n\ge 0}
\bigg[ \frac{q^{\lambda+n}}{1-zq^{\lambda + n}}-
\frac{q^{n}}{1-zq^{ n}} \bigg]
=-\frac{1-q^{\lambda}}{1-q} \times (1-q)\sum\limits_{n\ge 0}
\frac{{q^n}}{(1-zq^{\lambda + n})(1-zq^{ n})}.
\end{multline*}
Thus, as $q\to 1^-$ we have
$$
h_q(z) \to - \lambda \int_0^1 \frac{d y}{(1-zy)^2}=- \frac{\lambda}{1-z}.
$$
This convergence clearly holds pointwise, for all $z\in {\mathbb C}\setminus [1,\infty)$.
To prove uniform convergence on compact subsets, we use Montel's theorem, which requires us to show that the functions
$h_{q}$ are uniformly bounded on compact subsets of ${\mathbb C}\setminus [1,\infty)$ as $q\to 1^-$.
For small $\epsilon>0$, we define $D_{\epsilon}$ to be the closed domain
$$
D_{\epsilon}=\{z \in \c : |\arg(z)|\le \epsilon\} \setminus \{z \in \c : |z|<1-\epsilon\}.
$$
The domain $D_{\epsilon}$ is shown in Figure \ref{fig_D_epsilon}.

\begin{figure}[tb]

\usetikzlibrary{patterns}
   \begin{tikzpicture}[scale=3.5].

  \draw (5mm,0mm) arc [start angle=0, end angle=20, radius=5mm];
 \node[left] at (.5,0.075) { $\epsilon$};
  \draw[->] (0,-1.1) to (0,1.1);
 \draw[->] (-1.2,0) to (1.8,0);
 \node[left] at (0,1.1) {  $\i y$};
  \node[below] at (1.8,0) { $x$};
  \node at (1.2,.2) {   $D_{\epsilon}$};

  \draw [fill=black] (0,.7) circle [radius=0.03];
  \node[above] at (0.2,0.7) {  $1-\epsilon$};
     \draw [fill=black] (.7,0) circle [radius=0.03];
   \node[right] at (0.65,-.45) {  $1-\epsilon$};
     \draw[->,thick] (0.65,-.45) to [out=180,in=225] (0.65,-0.05);

     \draw[-,dashed] (0,0) to (0.657785, 0.239414);
       \draw[-,dashed] (0,0) to (0.657785, -0.239414);
 \draw [dashed] (0,0) circle [radius=0.7];
    \draw[-,dashed] (1,-.05) to (1,.05);
 \node[below] at (1.,0) {  $1$};

  \draw[thick] (0.657785, -0.239414) arc [start angle=-20, end angle=20, radius=.7];
  \draw[-,thick] (0.657785, -0.239414) to (1.59748, -0.581434);
   \draw[-,thick] (0.657785, 0.239414) to (1.59748, 0.581434);
                  \fill[pattern=north west lines, pattern color=gray!60] (0.657785, -0.239414) arc [radius=.7, start angle=-20, delta angle=40]
                  -- (1.59748, 0.581434) arc [radius=1.7, start angle=20, delta angle=-40]
                  -- cycle;
\end{tikzpicture}
 \caption{Domain $D_{\epsilon}$.
  \label{fig_D_epsilon}}
\end{figure}
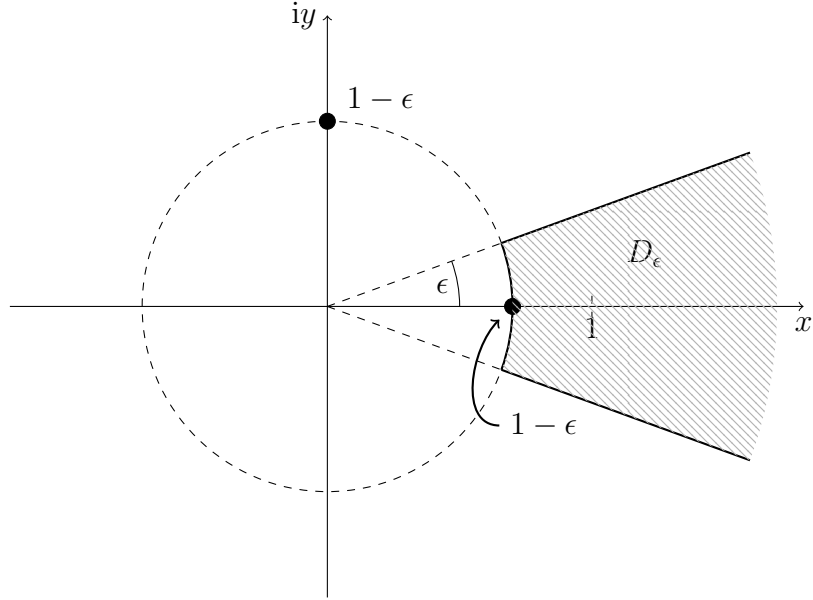
Now we take an arbitrary compact set $K \subset {\mathbb C} \setminus [1,\infty)$. There exists $\epsilon>0$ small enough such that $K \subset {\mathbb C} \setminus D_{\epsilon}$.
Since $q\in[0,1)$, for any $z\in K$ (in fact, for any $z\not\in D_\epsilon$) and any $n=0,1,\dots$ we have  $q^n z\in {\mathbb C} \setminus D_{\epsilon}\subset  {\mathbb C} \setminus D_{\epsilon/2}$.

Next, noting that $\lim_{q\to 1^-} q^\la=1$, there is   $q_0 \in (0,1)$ such that
$|q^{\lambda}|<1-\epsilon/2$  and $|\arg(q^\la)|<\epsilon/2$  for all $q \in (q_0, 1)$. This implies that for any  $z\not\in D_\epsilon$  we have   $q^\la z \in  {\mathbb C} \setminus D_{\epsilon/2}$. Indeed,
 if $z\not\in D_\epsilon$ then either $|z|<1-\epsilon$ or $|\arg(z)|>\epsilon$. In the first case, we have
 $|z q^\la|<(1-\epsilon)(1+\epsilon/2)<1-\epsilon/2$. In the second case,  $|\arg(q^\la z)|\geq |\arg(z)|-|\arg(q^\la)|>\epsilon/2$.

Therefore, we have shown that there exists $q_0$ such that if $z\in K$ then, with $q^n z\not\in D_\epsilon$, both  $q^n z$ and $q^{n+\la} z$ are in ${\mathbb C} \setminus D_{\epsilon/2}$ for all $n=0,1,\dots$ and all $q\in(q_0,1)$.
But the distance from $1$ to the set ${\mathbb C} \setminus D_{\epsilon/2}$ is strictly positive.
This means that there is $\delta>0$ such that for all  $z\in K$,
 $q\in (q_0,1)$  and every $n\ge 0$ we have
$$
|1-zq^{\lambda + n}|\ge \delta, \;\;\; |1-zq^{n}|\ge \delta,
$$
which implies an upper bound
$$
|h_q(z)|<\frac{|1-q^{\lambda}|}{|1-q|} \times (1-q)
\sum\limits_{n\ge 0} q^n \delta^{-2}= \frac{|1-q^{\lambda}|}{|1-q|} \times  \delta ^{-2}< C |\lambda| \delta^{-2},
$$
for some constant $C>0$.

Thus we have proved that the functions $h_q(z)$ converge to $-\lambda/(1-z)$ as
$q\to 1^-$, uniformly on compact subsets of ${\mathbb C}\setminus [1,\infty)$.
Then, integrating over  the segment from $0$ to $z$

$$
f_q(z)=\exp\Big( \int_0^z h_q(w) d w\Big) \to (1-z)^{\lambda}
$$
as $q\to 1^-$, also uniformly on compact subsets of ${\mathbb C}\setminus [1,\infty)$.
\end{proof}

Recall \cite[Section 1.10]{gasper2004basic} that for $z\in\CC\setminus\{0,-1,-2,\dots\}$ and $q\in[0,1)$,  the $q$-Gamma function is defined by

   \begin{equation}
    \label{q-Gamma}
    \Gamma_q(z) :=  (1-q)^{1-z} \;\frac{(q;q)_\infty}{(q^z;q)_\infty}.
    \end{equation}

The following result appears in  Ch. 16 of  Ramanujan notebook. For the statement, see  \cite[Entry 1]{adiga1985ramanujan},
or \cite[ p.  13, Entry 1(ii)]{berndt-ramanujan}. For the proof, see  \cite[Theorem B.2]{koornwinder1990jacobi}.
\begin{lemma}\label{L:Ram2} If $z\ne 0, -1,-2,\dots$ then
    \begin{equation}\label{qG2G}
       \lim_{q\nearrow 1}\Gamma_q(z)=\Gamma(z).
    \end{equation}

\end{lemma}
We will also need the following.

\begin{lemma}\label{Fact2} Let $q=e^{-2/M}$.
For any $c>0$ the function
$$
 t  \mapsto |\Gamma_q(c+\i t)|
$$
is periodic with period $M\pi$. It is increasing for $t\in [-M \pi/2,0]$ and decreasing for $t\in [0,M\pi/2]$.
\end{lemma}
\begin{proof}Referring to \eqref{q-Gamma} we note that
$$
|\Gamma_q(c+\i t)|=(1-q)^{1-c} (q;q)_{\infty} \prod\limits_{n\ge 0} \frac{1}{|1-q^{n+c+\i t}|}.
$$
As $t$ changes over the interval $[-M\pi/2, M\pi/2]$, the point $w=w(t)=q^{n+c+\i t}$ goes around a circle $|w|=q^{n+c}$, thus $1-w$ goes around a circle centered at $1$ and having radius $r=q^{n+c}$.
Thus, the function $$ t\in\RR \mapsto |1-w(t)|=\sqrt{1+q^{2(n+c)}-2q^{n+c}\cos(\tfrac{2t}{M})}$$ is periodic with period $M\pi$, has the minimum at $t=0$ and the maximum at $t=\pm M \pi/2$ and is clearly monotone between these points, which means that for every $n\ge 0$, the function
$$
t \in [-M\pi/2, M\pi/2]  \mapsto \frac{1}{|1-q^{n+c+\i t}|}
$$
is increasing on $[-M \pi/2,0]$ and decreasing on $[0,M\pi/2]$. Taking the product of positive increasing/decreasing functions gives us an increasing/decreasing function. \end{proof}

\subsection{Bounds on $\Gamma_q$}\label{Sec:Bds-Gq}
Following numerous references, such as  \cite{corwin-knizel2021stationary,Zhang2008} we will use the Jacobi theta functions to derive the bounds  on  $\Gamma_q$ that we need.

We assume that $v, \tau\in \c$, $\im(\tau)>0$ and we set $q:=e^{\pi \i \tau}$.
The Jacobi theta functions are defined as
\begin{align}
\label{def_theta1}
{\theta}_1(v|\tau)&=2 q^{1/4} \sum\limits_{n\ge 0} (-1)^{n} q^{n(n+1)} \sin((2n+1) \pi v),\\
\label{def_theta4}
{\theta}_4(v | \tau)&=1+2 \sum\limits_{n\ge 1} (-1)^n q^{n^2 }\cos(2 n \pi v)
.
\end{align}
We will need the following modular transformations of theta functions ${\theta}_1(v|\tau)$ and ${\theta}_4(v|\tau)$.
\begin{align}
\label{theta_1_4}
{\theta}_1(v|\tau)&= \i q^{1/4} e^{-\pi \i v} {\theta}_4(v-\tfrac{\tau}{2} | \tau),\\
\label{theta_1_over_tau}
{\theta}_1(v|\tau) &=\i \sqrt{\frac{\i}{\tau}}
e^{-\pi \i v^2/\tau} {\theta}_1( \tfrac{v}{\tau} | -\tfrac{1}{\tau}).
\end{align}
 The Jacobi's triple product identity for ${\theta}_1(v|\tau)$ is
  \begin{equation}
    \label{theta1_product}
{\theta}_1(v|\tau)=2 q^{1/4} \sin(\pi v) \prod\limits_{n\ge 1} (1-q^{2n})(1-q^{2n}e^{2\pi \i v})
(1-q^{2n}e^{-2\pi \i v}).
  \end{equation}
  That is,
    \begin{equation}
    \label{theta1_product+}
{\theta}_1(v|\tau)=2 q^{1/4} \sin(\pi v) (q^2,q^2 e^{2\pi \i v}, q^2e^{-2 \pi \i v};q^2)_\infty. \\
  \end{equation}

  All of these results can be found in \cite[Chapter 10]{Rademacher-1973}.

The next results also require the following elementary bounds:
for all $u\in \r$
\begin{equation}\label{sinh_bounds}
   e^{-|u|} \le  \frac{u}{\sinh(u)} \le  10 e^{-4|u|/5}
   \;\; {\textnormal{ and }} \;\;  \Big| \frac{u}{1-e^u} \Big| \le 1+|u|.
\end{equation}
We leave their proofs to the reader.

\begin{lemma}\label{lemma_q_Gamma_bound2}
Let $q=e^{-2/M}$. There exists $M^*>0$ and a universal constant $0<C<\infty$  such that
\begin{equation}\label{q_Gamma_bound2}
\frac{1}{C} e^{ -\frac{\pi}{2}|x|} <|\Gamma_q(1+\i x)|< C e^{-\frac{3 \pi}{20} |x|}
\end{equation}
for all   $M\ge M^*$ and all $x \in [-\pi M/2, \pi M/2]$.
\end{lemma}
\begin{proof}
In view of the product formula \eqref{theta1_product} combined  with \eqref{q-Gamma} we have the following representation
\begin{equation}
\label{Theta2Gammaq}
|\Gamma_q(1+ \i x)|^2  =
2 \sin(\tfrac{x}M)\frac{ e^{-\frac{1}{4M}} (q;q)_{\infty}^3}
{\theta_1\left(\frac{x }{ M \pi} \vert \frac{\i}{M \pi}\right)}.
\end{equation}
We use  \eqref{theta_1_over_tau} to rewrite  \eqref{Theta2Gammaq} as follows
$$
|\Gamma_q(1+ \i x)|^2  =\left(\frac{(q;q)_\infty\,\exp(\frac{\pi^2M}{12})}{\sqrt{M\pi}}\right)^3\,e^{-\frac{1}{4M}}\,\frac{\sin(x/M)}{x/M}\,\frac{2\sinh(\pi x)}{v(x,M)}\,\frac{\pi x}{\sinh(\pi x)}\,e^{\frac{x^2}{M}},
$$
where
\begin{equation}\label{vxm}
 v(x,M):=
- \i e^{ \pi^2 M/4} \theta_1(\i x | \i \pi M)=2\sinh(\pi x)\left(1+S(x,M)\right),
\end{equation}
with
$$
S(x,M):= \sum\limits_{n\ge 1} (-1)^{n}\times\frac{\sinh( (2n+1) \pi x)}{\sinh(\pi x)} \;e^{-\pi^2 M n(n+1)}.
$$

In view of  formula (35) in \cite{Zhang2008} with $\gamma=\pi$, $n=M^2$ and $a=1/2$ (or \cite[Proposition 2.3]{corwin-knizel2021stationary}) we have
$$
\frac{(q;q)_\infty\,\exp(\frac{\pi^2M}{12})}{\sqrt{M\pi}}\to 1 \quad \mbox{as $M\to\infty$}.$$
Consequently, for $\delta>0$ and $M$ large enough we have
\begin{equation}\label{qqM}
1-\delta<\left(\frac{(q;q)_\infty\,\exp(\tfrac{\pi^2M}{12})}{\sqrt{M\pi}}\right)^3\,e^{-\frac{1}{4M}}<1+\delta.
\end{equation}

For $x \in [-\frac{\pi M}2, \frac{\pi M}2]$, applying bounds \eqref{sinh_bounds}, we obtain:
\begin{align*}
&\bigg|\frac{\sinh((2n+1) \pi x)}{\sinh(\pi x)} \bigg|=
(2n+1) \bigg|\frac{\sinh( (2n+1) \pi  x)}{ (2n+1) \pi x} \bigg|
\times \bigg|\frac{\pi x}{\sinh(\pi x)} \bigg|
\\& \;\;\;  \le  (2n+1) \times e^{(2n+1)\pi|x|}
 \times  10  e^{-4 \pi |x|/5}
 \leq
 10 (2n+1) e^{\pi^2 M (n+ 1/2)}.
\end{align*}
 Consequently,
\begin{equation*}
|S(x,M)| \le 10 \sum\limits_{n\ge 1} (2n+1) e^{-\pi^2 M (n^2- 1/2)}
\leq
 10  e^{-\pi^2 M/2 } \times \Big[ \sum\limits_{n\ge 1} (2n+1) e^{-\pi^2 (n^2-1) }\Big] =
C_1 e^{-\pi^2M/2}
\end{equation*}
for an absolute constant $C_1>0$.
Therefore, for $\eps>0$ and for $M$ large enough we have
\begin{equation}\label{theta_estimate}
1-\eps\le \frac{v(x,M)}{2\sinh(\pi x)}\le 1+\eps,
\end{equation}
for all $x\in[-\pi M/2, \pi M/2]$.
In other words, we proved that
\begin{equation}\label{theta_estimate+}
- \i e^{ \pi^2 M/4} \theta_1(\i x | \i \pi M)=
2 \sinh(\pi x) (1+o(1)), \;\;\;
M\to \infty,\;\;
\end{equation}
uniformly in $x\in[-\pi M/2, \pi M/2]$.

Taking into account   \eqref{qqM}, \eqref{theta_estimate} and the trivial estimate
$$
\frac{2}{\pi}\le \frac{\sin(x/M)}{x/M}\le 1,\quad x \in [-\tfrac{\pi M}2, \tfrac{\pi M}2],
$$
we get
\begin{equation*}\label{deleps}
\frac{2}{\pi}\frac{1-\delta}{1+\eps}\,e^{x^2/M} \frac{\pi |x|}{|\sinh(\pi x)|}\le |\Gamma_q(1+\i x)|^2\le
\frac{1+\delta}{1-\eps}\,e^{x^2/M} \frac{\pi |x|}{|\sinh(\pi x)|},
\end{equation*}
which holds for  $M$ large enough and for all $x \in [-{\pi M}/2, {\pi M}/2]$. Setting   $\eps=\delta=1/2$ we get
 \begin{equation}\label{gamma_bound_proof1}
\frac{2}{3\pi} \frac{\pi |x|}{|\sinh(\pi x)|}<\frac{2}{3\pi}  e^{x^2/M} \frac{\pi |x|}{|\sinh(\pi x)|} <|\Gamma_q(1+\i x)|^2<3 e^{x^2/M} \frac{\pi |x|}{|\sinh(\pi x)|}.%
\end{equation}

In the range $x \in [-\pi M/2, \pi M/2]$ we have $x^2/M \le \pi |x|/2$. This fact, the  bounds \eqref{sinh_bounds} and \eqref{gamma_bound_proof1} imply that for all $M$ large enough
\begin{equation*}
\frac{2}{3\pi}  e^{-\pi |x|} <|\Gamma_q(1+\i x)|^2<30  e^{-\frac{3}{10} \pi |x|},%
\end{equation*}
from which the desired result \eqref{q_Gamma_bound2} follows by taking the square root.
\end{proof}

\begin{lemma}\label{lemma_q_Gamma_bound4}
Let $q=e^{-2/M}$ and assume that $\beta \in (\pi/4,3\pi/4)$.
 There exist $A>0$ and $M^*>0$ (which may depend on $\beta$) such that
\begin{equation}\label{ratio_Gammas_bound}
\frac{|\Gamma_q(1-\i \beta M+ \i x)|}{|\Gamma_q(1-\i \beta M)|}<A e^{\pi x/2}
\end{equation}
for all $M\ge M^*$ and $x\in [0,\pi M/2]$.
\end{lemma}

\begin{proof}
Expression \eqref{Theta2Gammaq} with modular transformation \eqref{theta_1_over_tau} applied twice gives
\begin{equation}\label{lemma4_proof1}
\frac{|\Gamma_q(1-\i \beta M+ \i x)|^2}{|\Gamma_q(1-\i \beta M)|^2}
=\frac{\sin(\beta-x/M)}{\sin(\beta)}\,
 \frac{\theta_1(\i \beta M | \i \pi M)}
{\theta_1(\i (\beta M-x) | \i \pi M)}\,\,e^{x^2/M - 2 \beta x}.
\end{equation}
 We consider two cases.  If $\beta \in (\pi/4,\pi/2]$, %
 then $1/\sin \beta\leq \sqrt{2}$ and the first arguments in both $\theta_1$ functions   satisfy   $|\beta M|\leq M\pi/2$ and  $|\beta M-x|\leq  M \pi/2$.
Using  estimate \eqref{theta_estimate+}   we obtain
$$
\frac{\theta_1(\i \beta M | \i \pi M)}
{\theta_1(\i (\beta M-x) | \i \pi M)}=  \frac{\sinh(\pi \beta M)}
{\sinh(\pi (\beta M -x))}(1+o(1)),\;\;\;\mbox{as }\;M\to\infty,
$$
 uniformly in $x\in[-\frac{M\pi}{2},\,\frac{M\pi}{2}]$.
We also write
\begin{multline*}
  \sin(\beta-\tfrac xM) \frac{\sinh(\pi \beta M)}
{\sinh(\pi (\beta M -x))}
=\frac{-1}{2\pi M}\times \frac{\sin(\beta-x/M)}{\beta-x/M}\times
\frac{2\pi (x-\beta M)} {1-e^{2\pi (x-\beta M)}} \times (1-e^{-2\pi \beta M})
\times e^{\pi x}.
\end{multline*}
The second and fourth factors on the right-hand side are bounded by 1 in absolute value, and using  the last bound in \eqref{sinh_bounds} we see that the product of the first and third factors is bounded by $1+\pi$.
Using these results and the bound $x^2/M \le \pi x/2$ on $x\in
[0,\pi M/2]$ we see that in the case $\beta \in (\pi/4, \pi/2]$
we have
$$
\frac{|\Gamma_q(1-\i \beta M+ \i x)|^2}{|\Gamma_q(1-\i \beta M)|^2}
<A_1 e^{\pi x/2-2\beta x+\pi x} =A_1e^{x(3\pi/2-2\beta)}\le A_1 e^{\pi x}
$$
for some absolute constant $A_1$, and this implies \eqref{ratio_Gammas_bound}.

Next we consider the case $\beta \in (\pi/2,3\pi/4)$.  We use formula \eqref{theta_1_4}, which together with \eqref{lemma4_proof1} gives %
\begin{equation}\label{theta1_theta4_ratio}
\frac{|\Gamma_q(1-\i \beta M+ \i x)|^2}{|\Gamma_q(1-\i \beta M)|^2}
=\frac{\sin(\beta-x/M)}{\sin(\beta)}\,
 \frac{\theta_4(\i \tilde \beta M | \i \pi M)}
{\theta_4(\i (\tilde \beta M-x) | \i \pi M)}\,e^{x^2/M +(\pi- 2 \beta) x}
\end{equation}
with $\tilde \beta=\beta-\pi/2 \in (0,\pi/4)$.

We use the same method as in the proof of Lemma \ref{lemma_q_Gamma_bound2}.  In view of \eqref{def_theta4} we write
\begin{equation}\label{theta4_est}
 \theta_4(\i w | \i \pi M)=1+\tilde S(w,M),
\end{equation}
and estimate
$$
\tilde S(w,M):=2 \sum\limits_{n\ge 1} (-1)^n e^{-\pi^2 M n^2 }\cosh(2 n \pi w)
$$
as follows: for any small $\eps>0$,  $M\ge 1$ and
$|w| \le (1-\eps) {M \pi}/2$ we have
\begin{multline*}
    |\tilde S(w,M)| \le 2 \sum\limits_{n\ge 1} e^{-\pi^2 M n^2 + 2n \pi |w|}
    \le 2 \sum\limits_{n\ge 1} e^{-\pi^2 M n^2 + n \pi^2  (1-\eps) M  }=
\\ =
2 e^{-\pi^2 M \eps}\sum\limits_{n\ge 1} e^{-\pi^2 M (n^2-n) - (n -1)\pi^2  \eps M }
\le 2 e^{-\pi^2 M \eps}\sum\limits_{n\ge 1} e^{-\pi^2 M (n^2-n) }\le\,C_2 e^{-\pi^2 M \eps }
\end{multline*}
for an absolute constant $C_2>0$,

From the above estimate and \eqref{theta4_est} it follows that for arbitrary $\eps\in(0,1)$, as $M \to +\infty$,
\begin{equation}\label{theta4_estimate}
 \theta_4(\i w | \i \pi M)=1+o(1),\quad \mbox{uniformly in}\;\;w\in [-(1-\eps)\tfrac{\pi M}2, (1-\eps)\tfrac{\pi M}2].
\end{equation}

Since  $\tilde \beta \in (0,\pi/4)$ and $x\in [0, \pi M/2]$, it is easy to see that
$$\tilde \beta M - x,\;\tilde \beta M\in \left[-(1-\eps)\tfrac{\pi M}2, (1-\eps)\tfrac{\pi M}2\right]\quad \mbox{with}\;\;\eps=2\tilde \beta/\pi.$$

Therefore,  applying  \eqref{theta4_estimate} for $w =\tilde \beta M - x $ and $w=\tilde \beta M$ to the second factor in \eqref{theta1_theta4_ratio} we conclude that
 \begin{align}\label{lemma4_proof7}
\frac{|\Gamma_q(1-\i \beta M+ \i x)|^2}{|\Gamma_q(1-\i \beta M)|^2}
&=
\frac{\sin(\beta-x/M)}{\sin(\beta)}
e^{x^2/M  +(\pi- 2 \beta) x} \times (1+o(1))
\end{align}
as $M\to\infty$ uniformly in $x\in [0,\pi M/2]$.  Since $\beta \in (\pi/2, 3\pi/4)$, the right-hand side in \eqref{lemma4_proof7} can be bounded by $C\exp(\pi x)$ for $x\in[0,\,\pi M/2]$  and we again conclude that
\eqref{ratio_Gammas_bound} holds.
\end{proof}

  \begin{lemma}\label{L:f_M}
  Fix $u\in\RR$ and let    $q_{M}=e^{-2/M}$, $M=1,2\dots$.
Then there exist constants $A,B,M_*>0$ that depend only on $u$ such that for all $M>M_*$ we have
\begin{equation}\label{GqGq}
   {\mathbf 1}_{\{|s|\leq M \pi/2\}} \left\vert\Gamma_{q_{M}}(1+\i (s+u))\right\vert\leq A e^{-B |s|}
\end{equation}
 for all $s\in \RR$.
  \end{lemma}

\begin{proof} %
We will show that \eqref{GqGq} holds with
$A=Ce^{3\pi|u|/20}$, %
$B=3\pi/40$ and $M_*=\max\{M^*,4|u|/\pi\}$, where %
$M^*$  and $C$ are
from  Lemma \ref{lemma_q_Gamma_bound2}.  Let $M>M_*$ and $|s|\leq M\pi/2$.

  We consider two cases.
 If $|s+u|\leq M\pi/2$,
 then \eqref{GqGq} follows from
 Lemma \ref{lemma_q_Gamma_bound2}.
 On the other hand, if $|s+u|> M\pi/2$, then with $M>M_*=\max\{M^*,4|u|/\pi\}\geq 4|u|/\pi$  we have $|u|\leq M\pi/4$. Recalling that  $|s|\leq M\pi/2$, we get $M\pi/2<|s+u|\leq M\pi/2+M\pi/4=3M\pi/4$.

By Lemma \ref{Fact2},  function $x\mapsto |\Gamma_q(1+\i x)|$ is periodic  with period $M\pi $  and is increasing  on the interval $[-M\pi/2,-M\pi/4]$. With $x=s+u\in[M\pi/2,3M\pi/4]$ we have $x-M\pi\in[-M\pi/2,-M\pi/4]$ so we get
$$ \left\vert\Gamma_q(1+\i x)\right\vert =\left\vert\Gamma_q(1+\i (x-M\pi))\right\vert \leq \left\vert\Gamma_q(1-\i M \pi/4)\right\vert \leq
C e^{-\frac{3\pi^2}{80} M}\leq  %
C  e^{-\frac{3\pi}{40}|s|},$$
where we used Lemma \ref{lemma_q_Gamma_bound2} and then the bound $M\ge   \tfrac2{\pi} |s|$. Thus in this case
\eqref{GqGq} follows, too.

\end{proof}
We also need the following result.%
\begin{lemma}\label{P:gamma}
 There exist constants $A,B,M_*>0$ such that with $q=e^{-2/M}$, for all $M>M_*$ and all $u\in\RR$ we have
  \begin{equation}
    \label{Gq}
      {\mathbf 1}_{\{|u|<M \pi/2\}} \frac{1}{|\Gamma_q(\i u)|^2}\leq A u^2 e^{B|u|}.
  \end{equation}
\end{lemma}
\begin{proof}  Definition \eqref{q-Gamma} yields
  $$\frac{1}{|\Gamma_q(\i u)|^2}= \frac{|1-q^{\i u}|^2}{(1-q)^2} \frac{1}{|\Gamma_q(1+\i u)|^2}.$$
 Since $|1-q^{\i u}|^2=4\sin^2(u/M)$, we get
  $$
  \frac{|1-q^{\i u}|^2}{(1-q)^2}=4u^2\,\frac{1}{(M(1-e^{-2/M}))^2}\,\left(\frac{\sin(u/M)}{u/M}\right)^2\leq C u^2.
  $$
Thus \eqref{Gq} follows from Lemma \ref{lemma_q_Gamma_bound2}.
\end{proof}
\subsection{Additional technical lemmas}
The following technical lemma is needed in the proof of Theorem \ref{Thm:3.2}.

\begin{lemma}\label{L:dc-2} Let $\wt a,\wt b$ be either real or complex conjugates.
If $\re(\wt a)\geq 0$,  $\re(\wt b)\geq 0$, $c\geq 0$ then for all $s \in\RR$ and all %$0\leq q<1$
$q\in[0,1)$
we have %
\begin{equation}
  \label{*dc-2}
  \left\vert\frac{ (-\wt a q^{1+c+\i s},-\wt b q^{1+c+\i s};q)_\infty} { (-\wt a q,-\wt b q;q)_\infty }\right\vert\leq 1.
\end{equation}
\end{lemma}
\begin{proof}We first prove the bound for the case of complex-conjugate parameters.
Write $\wt a=r e^{\i \alpha}$, $\wt b=r e^{-\i \alpha}$ with $\alpha\in[-\pi/2,\pi/2]$ and $r>0$. (The case $r=0$ is trivial.)
Consider one factor of the infinite product:

$$
\frac{|(1+r q^{n+c +\i  s }e^{\i \alpha})(1+r q^{n+c +\i  s }e^{-\i \alpha})|}
{|(1+r q^n e^{\i \alpha})(1+r q^n e^{-\i \alpha})|}=\frac{|e^{\i \alpha}+r q^{n+c +\i  s } |\cdot |e^{-\i  \alpha}+r q^{n+c +\i  s } |}
{|e^{\i \alpha}+r q^n |\cdot|e^{-\i \alpha}+r q^n |},\; n\geq 1.
$$
As $s $ varies through the real line, the point $z=-r q^{n+c +\i  s }$ moves on the circle $|z|=r q^{n+c}$. One readily checks that the  product in the numerator is  largest when $\re (z)$ is minimal, i.e., when $s =0$.
Thus
$$\frac{|(1+r q^{n+c +\i  s }e^{\i \alpha})(1+r q^{n+c +\i  s }e^{-\i \alpha})|}
{|(1+r q^n e^{\i \alpha})(1+r q^n e^{-\i \alpha})|}\leq
\frac{|(1+r q^{n+c}e^{\i \alpha})(1+r q^{n+c }e^{-\i \alpha})|}
{|(1+r q^n e^{\i \alpha})(1+r q^n e^{-\i \alpha})|}.$$
To end the proof, we note that
$$
\frac{|1+r q^{n+c}e^{\i \alpha}|^2}
{|1+r q^n e^{\i \alpha}|^2}\leq 1.
$$

The proof for the case of real $\wt a,\wt b \in [0,\infty)$ requires some minor modifications and is omitted. (This case is not used in the proof of Theorem \ref{Thm:3.2}.)
\end{proof}

The proof of Theorem \ref{Thm-q=1} uses the following  technical estimate:
\begin{lemma}\label{Prop-Alexey}
 Let $q=e^{-2/M}$. If $a=-q e^{\i\alpha}$, $b=-q e^{-\i \alpha}$  and $|\alpha|<\pi/2$
 then there exist $A>0$, $B>0$ and $M_*>0$   such that
 \begin{equation}\label{ab/qaqb}
  {\mathbf 1}_{\{|u|<\pi M/2 \}} \frac{(a,b;q)_{\infty}^2}{|(aq^{\i u/2},bq^{\i u/2};q)_{\infty}|^2}\leq A e^{B|u|}
 \end{equation}
 for all $M\geq M_*$ and all $u\in\RR$.
\end{lemma}
\begin{proof} Replacing $u$ by $-u$ if needed, without loss of generality we may assume $u\geq 0$.
 In view of \eqref{q-Gamma} the left-hand side of \eqref{ab/qaqb}
can be rewritten in terms of $q$-Gamma functions. We have
\begin{equation*}
  \label{q/aq}
  \frac{|(a;q)_{\infty}|}{|(aq^{\i u/2};q)_{\infty}|}=
\frac{|\Gamma_q(1-\i \beta M + \i u/2)|}{|\Gamma_q(1-\i \beta M)|},
\end{equation*}
where $\beta=(\alpha+\pi)/2\in(\pi/4,3 \pi/4)$.
Similarly,
$$ \frac{|(b;q)_{\infty}|}{|(bq^{\i u/2};q)_{\infty}|}=
\frac{|\Gamma_q(1-\i \beta' M + \i u/2)|}{|\Gamma_q(1-\i \beta' M)|},$$
where $\beta'=(\pi-\alpha)/2\in (\pi/4,3 \pi/4)$.

 Therefore inequality \eqref{ab/qaqb} follows from Lemma \ref{lemma_q_Gamma_bound4}.%
\end{proof}

 \subsection*{Funding}
 This work was supported by Simons Foundation [Award Number: 703475 to W.~B.];
 Natural Sciences and
Engineering Research Council of Canada [to A.~K.];
 National Science Center Poland [project no.  2023/51/B/ST1/01535 to J.~W.];
 Taft Research Center at the University of Cincinnati.

%%% BBL
\def\polhk#1{\setbox0=\hbox{#1}{\ooalign{\hidewidth
  \lower1.5ex\hbox{`}\hidewidth\crcr\unhbox0}}}

\end{document}